\documentclass[11pt]{article}
\usepackage[normalem]{ulem}
\usepackage{fullpage, amssymb, amsthm, amsmath, bm, verbatim}
\usepackage{cancel}
\usepackage{hyperref}
\usepackage{caption}
\newtheorem{thm}{Theorem}
\newtheorem{lem}[thm]{Lemma}
\newtheorem{cor}[thm]{Corollary}

\newtheorem{prop}[thm]{Proposition}

\newtheorem*{equiv-lem}{Equivalence Lemma}
\theoremstyle{definition}

\newtheorem{example}[thm]{Example}
\newtheorem{defn}[thm]{Definition}

\def\mad{\textrm{mad}}

\newcommand\vph{\varphi}

\RequirePackage{marginnote,hyperref}
\newcommand\chisl{\chi^{\star}_{\ell}}
\newcommand\chil{\chi_{\ell}}
\newcommand\chisc{\chi^{\star}_{c}}
\newcommand\chic{\chi_{c}}
\newcommand\tE{\tilde{E}}

\renewcommand\geq{\geqslant}
\renewcommand\ge{\geqslant}

\renewcommand\le{\leqslant}

\renewcommand\P{\mathcal{P}}
\newcommand\B{\mathcal{B}}
\newcommand\G{\mathcal{G}}
\newcommand\swap{\textrm{swap}}
\newcommand\id{\textrm{id}}

\usepackage{tikz}
\usetikzlibrary{positioning}
\usetikzlibrary{decorations.pathmorphing}

\tikzstyle{uStyle}=[shape = circle, minimum size = 4pt, inner sep = 1pt,
outer sep = 0pt, fill=white, semithick, draw]
\tikzstyle{ugStyle}=[shape = circle, minimum size = 4pt, inner sep = 1pt,
outer sep = 0pt, fill=gray!70!white, semithick, draw]
\tikzstyle{ubStyle}=[shape = circle, minimum size = 4pt, inner sep = 1pt,
outer sep = 0pt, fill=black, semithick, draw]
\tikzstyle{uBstyle}=[shape = circle, minimum size = 9pt, inner sep = 1pt,
outer sep = 0pt, fill=white, semithick, draw]
\tikzstyle{lStyle}=[shape = circle, minimum size = 5pt, inner sep =
0.5pt, outer sep = 0pt, draw=none]

\title{List Packing and Correspondence Packing of Planar Graphs}
\author{Daniel W. Cranston\thanks{%
Department of Computer Science, Virginia Commonwealth
University, Richmond, VA, USA;
\texttt{dcranston@vcu.edu}
}
\and
Evelyne Smith-Roberge\thanks{%
School of Mathematics, Georgia Tech,
Atlanta, GA, USA;
\texttt{esmithroberge3@gatech.edu}
}
}

\begin{document}
\maketitle
\abstract{
    For a graph $G$ and a list assignment $L$ with $|L(v)|=k$ for all $v$, an $L$-packing
    consists of $L$-colorings $\vph_1,\cdots,\vph_k$ such that $\vph_i(v)\ne\vph_j(v)$ for
    all $v$ and all distinct $i,j\in\{1,\ldots,k\}$.  Let $\chisl(G)$ denote the smallest
    $k$ such that $G$ has an $L$-packing for every $L$ with $|L(v)|=k$ for all $v$.
    Let $\P_k$ denote the set of all planar graphs with girth at least $k$.
    We show that (i) $\chisl(G)\le 8$ for all $G\in \P_3$ and (ii) $\chisl(G)\le 5$ for all
    $G\in \P_4$ and (iii) $\chisl(G)\le 4$ for all $G\in \P_5$.  Part (i) makes progress on 
    a problem of Cambie, Cames van Batenburg, Davies, and Kang.  
    We also consider the analogue of $\chisl$ for 
    correspondence coloring, $\chisc$.  In fact, all 
    bounds stated above for $\chisl$ also hold for $\chisc$.
    \smallskip

}

\section{Definitions and Introduction}
\label{defn-sec}
Graph coloring is one of the most widely studied areas within discrete mathematics,  and one of its most popular variants is \emph{list-coloring}.  Each vertex $v$ in a graph $G$ is given a list $L(v)$ of allowable colors, and we seek an \emph{$L$-coloring}, which
is a proper coloring $\vph$ such that $\vph(v)\in L(v)$ for all $v\in V(G)$.  If $|L(v)|=k$ for all $v$, then $L$ is a \emph{$k$-assignment}.  And the \emph{list-chromatic number}, $\chil(G)$, is the minimum $k$ such that $G$ admits an $L$-coloring
for every $k$-assignment $L$. But what if we seek multiple $L$-colorings, all pairwise disjoint?
Given a graph $G$, a positive integer $k$, and a $k$-assignment $L$, an \emph{$L$-packing} 
is a set $\vph_1,\ldots,\vph_k$ of $L$-colorings such that $\vph_i(v)\ne \vph_j(v)$ for all $v\in V(G)$ 
and all distinct $i,j\in[k]$; here $[k]:=\{1,\ldots,k\}$.  Intuitively, $\vph$ partitions $L$ into $k$ $L$-colorings or, 
equivalently, packs $k$ $L$-colorings into $L$.

In this paper we study the \emph{list packing number} of $G$, denoted $\chisl(G)$, which is the minimum $k$ such that $G$ has a list packing for every $k$-assignment $L$.  It is not immediately 
clear that $\chisl$ is well-defined, but this is proved in~\cite{CCvBDK}, where the parameter was first studied; ~\cite{mudrock-note} also gives a short proof. 
In particular, always $\chisl(G)\le |G|$.  Furthermore, if $\chisl(G)=k$, then $G$ has an $L$-packing for every 
$s$-assignment $L$ with $s\ge k$.\footnote{We use induction on $s$, with the base case $s=k$.  So suppose that $s\ge k+1$.  Since 
$\chil(G)\le \chisl(G)<s$, clearly $G$ has an $L$-coloring $\vph_s$.  Now let $L'(v):=L(v) \setminus \vph_s(v)$ for all $v$.
By the induction hypothesis, $G$ has an $L'$-packing $\vph_1,\ldots,\vph_{s-1}$; this $L'$-packing combines with $\vph_s$ to 
give the desired $L$-packing.}

We also study a stronger notion, based on correspondence coloring.  A \emph{$k$-correspondence assignment} (or 
\emph{$k$-cover} for short) for a 
graph $G$ is an orientation $\vec{D}$ of $G$ and a map $\sigma:\vec{D}\to S_k$ that assigns each arc a permutation of $[k]$.  A 
$(\vec{D},\sigma)$-coloring is a map $\vph:V(G)\to [k]$ such that $\sigma(vw)(\vph(v))\ne \vph(w)$ for all $vw\in \vec{D}$.
The correspondence chromatic number, $\chic(G)$, is the minimum $k$ such that $G$ admits a $(\vec{D},\sigma)$-coloring for every 
choice\footnote{In fact, the choice of $\vec{D}$ is immaterial and just for easy reference, since the problem remains equivalent
when we reverse an edge $vw$ and replace $vw$ and $\sigma(vw)$ with $wv$ and $\sigma^{-1}(vw)$.}
of $\vec{D}$ and $\sigma$.

Similar to above, we seek many \emph{disjoint} $(\vec{D},\sigma)$-colorings.
For a graph $G$, integer $k$, and $k$-cover $(\vec{D},\sigma)$, a \emph{$(\vec{D},\sigma)$-packing} consists of functions $\vph_1,
\ldots,\vph_k$ such that each $\vph_i$ is a $(\vec{D},\sigma)$-coloring and $\vph_i(v)\ne \vph_j(v)$ for all distinct $i,j\in[k]$.
The \emph{correspondence packing number}, denoted $\chisc(G)$, is the minimum $k$ such that $G$ admits a $(\vec{D},\sigma)$-packing
for every $k$-correspondence assignment $(\vec{D},\sigma)$.  It is straightforward to verify\footnote{Given a $k$-assignment $L$, for each vertex $v$, we map $L(v)$ arbitrarily onto $[k]$ (for that $v$) and then, for each edge $vw$ in $E(G)$ we add a matching between the $[k]$ for $v$ and the $[k]$ for $w$ corresponding to the colors that appear in both $L(v)$ and $L(w)$; finally, we extend this matching arbitrarily to a 1-factor between these sets $[k]$.  Every correspondence coloring for the resulting $k$-cover will also be an $L$-coloring.} that $\chil(G)\le \chic(G)$ and $\chisl(G)\le \chisc(G)$ for all $G$.  In this paper, we focus mainly on $\chisc$.  

The notions $\chisc$ and $\chisl$ were first studied in~\cite{CCvBDK}, where the authors posed numerous open questions.
In particular, they highlighted the following: Determine $\max_{G\in \P}\chisl(G)$, where $\P$ is the class of all planar graphs.  
They proved the easy bound $\chisl(G)\le \chisc(G)\le 10$.  Our strongest result improves this significantly.  

\begin{thm}
    \label{P3-thm}
    $\chisc(G)\le 8$ for all planar graphs $G$.
\end{thm}

We also consider planar graphs of higher girth.

\begin{thm}
    \label{P4-thm}
    $\chisc(G)\le 5$ for all planar graphs $G$ with girth at least 4.  
\end{thm}

\begin{thm}
    \label{P5-thm}
    $\chisc(G)\le 4$ for all planar graphs $G$ with girth at least 5, and the value 4 is optimal. 
\end{thm}

We present these proofs in order of increasing difficulty.  
Theorems~\ref{P4-thm}, \ref{P5-thm}, and~\ref{P3-thm} are restated and proved, respectively, at the ends of Sections~\ref{P4-sec}, \ref{P5-sec}, and~\ref{P3-sec}.  In Section~\ref{matching-sec} we provide some lemmas needed in Section~\ref{P3-sec}.  Finally, in Section~\ref{open-sec}, we conclude with some open problems.
\bigskip

As we neared the end of preparing this manuscript, Stijn Cambie, Wouter Cames van Batenburg, and Xuding Zhu~\cite{CCvBZ} posted to arXiv a manuscript titled ``Disjoint list-colorings for planar graphs''.
This paper contains independent proofs of Theorems~\ref{P3-thm}--\ref{P5-thm}.  Those authors proved Theorems~\ref{P3-thm} and~\ref{P5-thm} somewhat more generally than we did.  In particular, they proved that $\chisc(G)\le 8$ for all graphs $G$ with maximum average degree less than 6.  However, parts of their proofs rely on computer verification.  In contrast, our proofs are completely
hand-checkable.

\subsection{Key Tools and Examples}
\label{outerplanar-sec}
Before proving our main results, it is helpful to first introduce our key tools and to illustrate their
application with a few examples.

\begin{defn}
    \label{bigraph-defn}
    A \emph{bigraph} is a bipartite graph with parts $A$ and $B$ such that $|A|=|B|$.  An \emph{$(s,t)$-bigraph}
    is a bigraph $H$ with $|A|=s$ and $\delta(H)\ge t$.
    Fix a graph $G$, a $k$-cover $(\vec{D},\sigma)$ of $G$, a vertex $v\in V(G)$, and a $(\vec{D},\sigma)$-packing $\vph$ of $G-v$.
    Form an auxiliary bigraph $H_v$ with $|A|=|B|=k$, with vertices in $A$ denoting colors in $[k]$ and vertices in $B$
    denoting $\vph_1,\ldots,\vph_k$ in $\vph$, and with edge $i\vph_j\in E(H)$ if the possibility $\vph_j(v)=i$ is
    permitted by $\vph$.  
    An \emph{obstruction} in a bigraph is $X\subseteq A$ such that $|N(X)|<|X|$.    (We only define
    $H_v$ formally in the case of correspondence packing, but the definition is nearly identical for list packing.)
\end{defn}

Note that extensions of $\vph$ to $v$ are in bijection with 
1-factors in $H_v$.  Throughout this paper we rely heavily on
Hall's Theorem, which says that a bigraph has a 1-factor if and 
only if it has no obstruction. 
When applying Hall's Theorem to an $(s,t)$-bigraph, the following proposition is useful for restricting the size of any possible obstructions.

\begin{prop}
    If $G$ is an $(s,t)$-bigraph and there exists $X\subseteq A$ with $|X|>|N(X)|$, then $t+1\le |X|\le s-t$.
    In particular, every $(2t,t)$-bigraph has a 1-factor.
    \label{easy-prop}
\end{prop}
\begin{proof}
    If $1\le|X|\le t$, then $|N(X)|\ge \delta(G)\ge t \ge |X|$.  Instead suppose $|X|\ge s-t+1$.  For each $w\in B$,
    by Pigeonhole $|N(w)\cap X|\ge d(w)+|X|-|A| \ge t + (s-t+1) - s = 1$.  Thus, $w\in N(X)$; 
    so $|N(X)| = |B| = s \ge |X|$.  For a $(2t,t)$-bigraph, $t+1\le |X|\le t$; so no $X$ exists.
\end{proof}

Proposition~\ref{easy-prop}, yields the following corollary, which was proved in \cite{CCvBDK} (see Theorem~9 therein).  
\begin{cor}\label{cor:degen}
    If $\chisc(G-v)\le k$ and $d(v)\le k/2$, then also $\chisc(G)\le k$.  In particular, if $G$ is $d$-degenerate, 
    then $\chisc(G)\le 2d$.
    \label{degen-cor}
\end{cor}
\begin{proof}
    For the first statement, note that $H_v$ is a $(k,k-k/2)$
    -bigraph.  Thus, by Proposition~\ref{easy-prop}
    the graph $H_v$ has no obstruction, so has a 1-factor by Hall's Theorem.  The second statement follows from the
    first by induction on $|G|$, with $k:=2d$.
\end{proof}

The following example first appeared in~\cite[Theorems~1~and~2]{CCvBDK2}.  But we present it here as a warm-up to familiarize 
the reader with the use of Definition~\ref{bigraph-defn} and Proposition~\ref{easy-prop}.

\begin{example}\label{ex:cycles}
    $\chisl(G)=\chisc(G)=2$ if $G$ is a forest with at least one edge, but $\chisl(C_k)=3<4=\chisc(C_k)$ for all $k\ge 3$.
    (In fact $\chisc(C_k)\le 4$ by Corollary~\ref{degen-cor}, but here we give a few more details.)
\end{example}
\begin{proof}
    If $G$ has an edge, then clearly $\chisc(G)\ge \chisl(G)\ge \chi(G)\ge 2$.  To prove $\chisc(G)\le 2$, we use induction
    on $|G|$.  Let $v$ be a vertex of $G$ with at most one neighbor.  By hypothesis, $\chisc(G-v)\le 2$.  Now we can extend
    the packing to $v$ because $H_v=K_{2,2}-M$ for some matching $M$; thus $H_v$ contains a 1-factor.

    We first show that $\chisc(C_k)\ge4$.  We denote the vertices of $C_k$, listed in cyclic order, by $v_1,\cdots, v_k$.  Consider the following 3-correspondence assignment.  Orient the graph as $v_iv_{i+1}$ for all $i\in [k-1]$ and $v_kv_1$.  Take the permutations
    $\sigma(v_iv_{i+1})=\id$, the identity, for all $i\in[k-1]$, and let $\sigma(v_kv_1)$ have a single transposition.  By symmetry,
    we assume that $\vph_i(v_1)=i$ for all $i\in [3]$.  It is easy to check, by induction on $j$, that $(\vph_1(v_j),\vph_2(v_j),
    \vph_3(v_j))$ must be an \emph{even} permutation for all $j\in [k]$.  However, now the edge $v_kv_1$ creates a conflict.
    Thus, $\chisc(C_k)>3$.  

    Next we prove that $\chisc(C_k)\le 4$.  We build the packing inductively on $i$; at each step but the last, this is easy 
    since $H_{v_i}$ is a $(4,3)$-bigraph.
    Finally, $H_{v_k}$ is a $(4,2)$-bigraph; by Proposition~\ref{easy-prop}, with $s=4$ and $t=2$, the graph $H_{v_k}$ has 
    no obstruction, so we can extend the packing to $v_k$.

    Now we show that $\chisl(C_k)\ge 3$.  If $k$ is odd, then $\chisl(C_k)\ge \chi(C_k)\ge 3$.  If $k$ is even, then let 
    $L(v_i)=\{1,2\}$ for all $i\in[k-2]$ and $L(v_{k-1})=\{1,3\}$ and $L(v_k)=\{2,3\}$.  It is easy to check that no
    $L$-coloring $\vph_i$ has $\vph_i(v_1)=2$.  Thus, $C_k$ has no $L$-packing.

   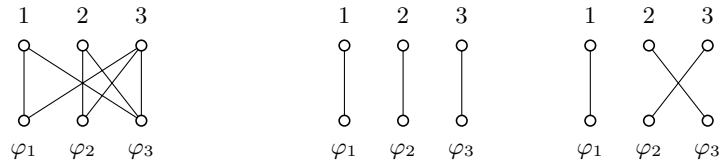
\begin{figure}[!h]
    \centering
\begin{tikzpicture}[xscale = .78, yscale=1]
    \def\off{.4cm}
\tikzset{every node/.style=uStyle}

\begin{scope}[xshift=.35in]
    \foreach \i in {1, 2, 3}
    {
        \draw[thick] (\i,1) node (a\i) {} (\i,0) node (b\i) {};
        \draw (a\i) ++ (0,\off) node[lStyle] {\footnotesize{${\i}$}};
        \draw (b\i) ++ (0,-\off) node[lStyle] {\footnotesize{$\vph_{\i}$}};
    }
    \draw (a3) -- (b3) -- (a2) -- (b2) -- (a3) -- (b1) -- (a1) -- (b3);
\end{scope}

\begin{scope}[xshift = 2.5in]
    \foreach \i in {1, 2, 3}
    {
        \draw[thick] (\i,1) node (a\i) {} (\i,0) node (b\i) {};
        \draw (a\i) ++ (0,\off) node[lStyle] {\footnotesize{${\i}$}};
        \draw (b\i) ++ (0,-\off) node[lStyle] {\footnotesize{$\vph_{\i}$}};
    }
    \draw (a1) -- (b1) (a2) -- (b2) (a3) -- (b3);
\end{scope}

\begin{scope}[xshift=4.15in]
    \foreach \i in {1, 2, 3}
    {
        \draw[thick] (\i,1) node (a\i) {} (\i,0) node (b\i) {};
        \draw (a\i) ++ (0,\off) node[lStyle] {\footnotesize{${\i}$}};
        \draw (b\i) ++ (0,-\off) node[lStyle] {\footnotesize{$\vph_{\i}$}};
    }
    \draw (a1) -- (b1) (a2) -- (b3) (a3) -- (b2);
\end{scope}

\end{tikzpicture}
       \caption{Left: $H_{v_k}$.  Center and right: 1-factors $M'_k$ and $M''_k$ in $H_{v_k}$.\label{cycle-fig}}
\end{figure}

    Finally, we show that $\chisl(C_k)\le 3$.  We fix a 3-assignment $L$ and construct an $L$-packing $\vph_1, \vph_2, \vph_3$ (simply $\vph$ for short).  
    If $L$ is everywhere identical, then we are done since $\chi(C_k)\le 3$.  So assume that $L(v_k)\ne L(v_1)$.  
    We can easily color $v_1,\ldots, v_{k-2}$ in order, since at each step $H_{v_i}$ is a $(3,2)$-bigraph.  
    Now $H_{v_k}$ (before coloring $v_{k-1}$) is a $(3,2)$-bigraph with $|E(H_k)|\ge 7$, since $L(v_1)\ne L(v_k)$.  By symmetry, assume that $L(v_k)=\{1,2,3\}$ and $H_{v_k}\supseteq K_{3,3}-\{1\vph_2,2\vph_1\}$; see Figure~\ref{cycle-fig}.
    Let $M'_k:=\{1\vph_1,2\vph_2,3\vph_3\}$ and $M''_k:=\{1\vph_1,2\vph_3,3\vph_2\}$.  Note that any matching $M_{k-1}$ in $H_{v_{k-1}}$ that forbids both $M'_k$ and $M''_k$ must include edge $1\vph_1$.  However, $H_{v_{k-1}}$ is a $(3,2)$-bigraph.  So we can choose an arbitrary edge $i\vph_1$, with $i\ne 1$, and find a 1-factor in
    $M_{k-1}$ containing $i\vph_1$ that extends $\vph$ to $v_{k-1}$.  Afterward, we extend $\vph$ to $v_k$ with either $M'_k$ or $M''_k$.
\end{proof}

Many proofs in this paper assume the existence of a minimum counterexample $G$ with cover $(\vec{D}, \sigma)$ to a correspondence packing theorem. We then delete a vertex, $v$; obtain a correspondence packing $\vph$ of $G-v$; and deduce some structure of the auxiliary bigraph $H_v$. To extend the correspondence packing $\vph$ to $v$, we often need to first modify the packing of nearby vertices. In nearly every case, the subgraph of $G$ induced by $v$ and the vertices where we modify the packing is a tree. In these cases, it is convenient to assume that in the cover of $G$, we have $\sigma(uv) = \id$ for each arc $uv$ in the tree. The following lemma justifies this assumption. This step
is standard in the study of correspondence coloring (see e.g.~\cite{dvovrak2018correspondence, dahlberg2023polynomial}).

\begin{lem}\label{lem:straightening}
    Let $(\vec{D},\sigma)$ be a $k$-cover of a graph $G$, and let $T \subseteq G$ be a tree that inherits the orientation $\vec{D}$. There exists a $k$-cover $(\vec{D}, \sigma')$ of $G$ such that $\sigma'(uv) = \id$ for every arc $uv \in E(T)$ and such that $G$ has a $(\vec{D}, \sigma')$-packing if and only if $G$ has a $(\vec{D},\sigma)$-packing.
\end{lem}

Essentially, we repeatedly grow the subtree $T'$ of ``straightened'' arcs to include an arc to a vertex $v$ previously outside $T'$; permute the colors at $v$ to straighten this new arc; and then update correspondingly the matchings for other arcs incident to $v$.

\begin{proof}
    We describe the $k$-cover $(\vec{D}, \sigma')$ below; that it is equivalent to $(\vec{D}, \sigma)$ follows easily by observation. We proceed by induction on $\ell$, the number of edges of $T$. If $\ell = 0$ there is nothing to prove. We assume then that $\ell \geq 1$ and that the result holds for all trees with fewer edges. Let $v$ be a leaf of $T$, and $u$ be its neighbor in $T$. We may assume $uv \in \vec{D}$; otherwise, we let $\sigma(uv) = \sigma^{-1}(vu)$. By induction, we may assume $\sigma(xy) = \id$ for every arc $xy$ in $E(T) \cap \vec{D}$ that is not $uv$. For each arc $xy$ in $\vec{D}$ that is not incident to $v$, let $\sigma'(xy) := \sigma(xy)$.  Let $\sigma'(uv) := \id$. Let $w$ be a vertex other than $u$ adjacent to $v$ in $E(G)$; as before, we may assume without loss of generality that $vw \in \vec{D}$. Let $\sigma'(vw) := \sigma(vw) \circ\sigma(uv)$ for each such arc $vw$.
\end{proof}

\section{Planar Graphs with Girth at least 4 and Lists of Size 5}
\label{P4-sec}

In this section, we prove that $\chisc(G)\le 5$ for every planar graph $G$ with girth at least 4.
Before proving this result, we present a few lemmas that we will need about matchings in bigraphs.

\begin{lem}
    \label{matching-lem}
    Fix a positive integer $k$.
    \begin{enumerate}
        \item[(1)] If $G$ is a $(2k+1,k+1)$-bigraph, then for all $e\in E(G)$ there exists a 1-factor containing $e$.
        \item[(2)] If $G$ is a $(2k+1,k)$-bigraph that has no 1-factor, then there exist $X\subseteq A$ and $Y\subseteq B$ 
            with $|X|=|Y|=k+1$ and $E(X,Y)=\emptyset$.  Furthermore, such $X$ and $Y$ are unique and $X$ is complete to 
            $B\setminus Y$ and also $Y$ is complete to $A\setminus X$.
    \end{enumerate}
\end{lem}

\begin{proof}

    To prove (1), consider an arbitrary edge $ab$ and let $G':=G-\{a,b\}$.  By Proposition~\ref{easy-prop}, 
    we find a 1-factor $M'$ in $G'$, which is a $(2k,k)$-bigraph.  So $M'+ab$ is the desired 1-factor in $G$.

    Now we prove (2).  By Hall's Theorem, if $G$ has no 1-factor, then there exists $X\subseteq A$ such that $|N(X)|<|X|$.  
    By Proposition~\ref{easy-prop}, we know $|X|=k+1$.  
    Since $\delta(G)\ge k$, we see that $X$ is complete to $N(X)$.  Let $Y:=B\setminus N(X)$.
    Note that $|N(Y)|\le |A\setminus X| = 2k+1-(k+1)=k$, so $Y$ must be complete to $N(Y)$, since $\delta(G)\ge k$.  And by definition,
    $E(X,Y)=\emptyset$. 
    Finally, we prove uniqueness.  Consider  $X'\subseteq A$ with $|X'|=k+1$ and $X'\ne X$.  Now there exist $x_1,x_2\in X'$ 
    with $x_1\in X$ and $x_2\notin X$.  But now $|N(X')|\ge |N(x_1)|+|N(x_2)\setminus N(x_1)| = |N(X)|+|B\setminus N(X)| = k+(k+1)$.  
    So $X$ is unique, as is $Y$ by symmetry.
\end{proof}

\begin{lem}
    \label{1gives2-lem}
    Fix an integer $k$ with $k\ge 2$.  Let $G$ be a $(2k+1,k)$-bigraph.  If $G$ has a 1-factor, then there 
    exists $v'\in A$ such that for all $v\in A-v'$ there exist 1-factors $M_1^v$ and $M_2^v$ that contain 
    distinct edges incident with $v$.
\end{lem}
\begin{proof}
    If each vertex $v\in A$ has at least two incident edges that appear in 1-factors, then we are done.  So assume instead 
    that there exists $v'\in A$ and edge $e$ incident with $v'$ such that $G-e$ has no 1-factor.  Let $G':=G-e$.  Now we
    essentially repeat the proof of Lemma~\ref{matching-lem}(2) applied to $G'$.  The main difference is that if $w'$ is the 
    other endpoint of $e$, then we have $d_{G'}(x)\ge k$ for all $x\in V(G)\setminus\{v',w'\}$, but only $d_{G'}(v')\ge k-1$ 
    and $d_{G'}(w')\ge k-1$.  Fortunately, that proof is robust enough that we still reach the following conclusion.
    There exists $X\subseteq A$ and $Y\subseteq B$ with $|X|=|Y|=k+1$ and $E_{G'}(X,Y)=\emptyset$.  Since $G$ has a 1-factor, 
    we conclude that $v'\in X$ and $w'\notin N_{G'}(X)$.  Thus, $G-\{v',w'\}$ consists of two copies of $K_{k,k}$.
    In this graph, every edge lies in many (in fact, $k!(k-1)!$) 1-factors.  This proves the lemma.
\end{proof}

\begin{lem}
    \label{k,k+1-lem}
    Fix an integer $k$ with $k\ge 3$.  Fix a graph $G$ and a $(2k-1)$-cover $(\vec{D},\sigma)$ such that $G$ has no $(\vec{D},\sigma)$-packing.  
    Now $G$ has no edge $vw$ with $d(v)=k$ and $d(w)\le k+1$ such that $G-v$ has a $(\vec{D},\sigma)$-packing.
\end{lem}
\begin{proof}
    Assume the contrary and let $(G,(\vec{D},\sigma))$, with $v,w\in V(G)$, be a counterexample; here $d(v)=k$ and $d(w)\le k+1$. 
    By hypothesis, $G-v$ has a $(\vec{D},\sigma)$-packing, $\vph$.  Assume we cannot extend $\vph$ to $v$.  
    By Lemma~\ref{matching-lem}(2), in $H_v$ there exist $X\subseteq A$ and $Y\subseteq B$ with 
    $|X|=|Y|=k$ and $E_{H_v}(X,Y)=\emptyset$.  By symmetry, we assume that $X=\{\vph_1,\ldots,\vph_k\}$ and 
    $Y=\{k,\ldots,2k-1\}$.  Further, by symmetry, we assume that $\vph$ uses color $i+k-1$ on $v$ for each $i\in[k]$.
    
    By Lemma~\ref{1gives2-lem} and symmetry, we assume 
    we can repack $w$ so that $\vph'_1$ avoids color $k$ at $w$; here $\vph'$ denotes this new $(\vec{D},\sigma)$-packing 
    of $G-v$.  Let $H'_v$ be the resulting new auxiliary graph for $v$.  If we can find a 1-factor in $H'_v$, then we 
    extend $\vph'$ to $v$ and we are done; so assume not.  

    By Hall's Theorem, there exists $X'\subseteq A$ such that $|N_{H'}(X')|<|X'|$; 
    by Proposition~\ref{easy-prop},  we have $|X'|=k$.  First, suppose that $X'=X$.  Now $|N_{H'_v}(X')|\ge 
    |N_{H_v}(X')|+|\{k\}| = (k-1)+1=|X'|$.  (Recall that each vertex of $N_{H_v}(X)$ loses at most one neighbor in $X$ when we 
    repack $w$; but in $H_v$ we had $X$ complete to $N_{H_v}(X)$.)  So assume instead that $X'\ne X$; thus, there exist $x',y'\in X'$
    such that $x'\in X$ and $y'\notin X$.  Now $|N_{H'_v}(X')|\ge |N_{H'_v}(x')|+|N_{H'_v}(y')\setminus N_{H'_v}(x')| \ge ((k-1)-1) 
    + k-1 \ge k = |X'|$, since $k\ge 3$.  Thus, no such $X'$ exists, a contradiction.
\end{proof}

Now we can prove the main result of this section.

\begin{thm}
    If $G$ is planar with girth at least 4, then $\chisc(G)\le 5$.
    More generally, $\chisc(G)\le 5$ for every graph $G$ with maximum average degree less than 4.
\end{thm}
\begin{proof}
    If $G$ is planar with girth at least $g$, then by Euler's formula
    $G$ has average degree less than $2g/(g-2)$; so if $g=4$, then this is average is less than $2(4)/(4-2)=4$.  
    Thus, the first statement follows from the second, which we now prove.
    Assume the statement is false and let $(G,(\vec{D},\sigma))$ be a pair witnessing this; among all such pairs, 
    choose $G$ to minimize $|G|$.
    
    By Corollary~\ref{degen-cor}, we know that $\delta(G)\ge 3$.  By Lemma~\ref{k,k+1-lem} with $k=3$, we know that $G$ has no 
    edge $vw$ with $d(v)=3$ and $d(w)\le 4$.  We will also show below that $G$ has no 5-vertex adjacent to four or more 3-vertices.
    Assuming that such a 5-vertex is forbidden, we can conclude the proof with the following easy discharging argument.
    
    To reach a 
    contradiction, we will show that if $G$ has none of the forbidden configurations listed above, then $G$ has average degree
    at least 4, contradicting the hypothesis.  We give to each vertex an initial charge equal to its degree.  We use a single discharging rule: Each 3-vertex takes
    charge $1/3$ from each neighbor.  By the previous paragraph, $\delta(G)\ge 3$ and 3-vertices cannot be adjacent.  Thus, each 
    3-vertex finishes with charge $3+3(1/3)=4$.  As noted above, each 4-vertex has no adjacent 3-vertex.  Thus, each
    4-vertex starts and ends with charge 4.  As we show below, each 5-vertex $v$ has at most three adjacent 3-vertices.  
    Thus, $v$ finishes with charge at least $5-3(1/3)=4$.  Finally, if $d(v)\ge 6$, then $v$ finishes with charge at least 
    $d(v)-d(v)/3 = 2d(v)/3 \ge 2(6)/3 = 4$.  Since every vertex finishes with charge at least 4, the average degree of $G$ is at 
    least 4, contradicting our hypothesis.  This completes the proof.

    Now we need only to show that $G$ has no 5-vertex $v$ adjacent to at least four 3-vertices.  
    Assume to the contrary that $G$ has a 5-vertex $v$ with four adjacent 3-vertices $w_1,w_2,w_3,w_4$.
    By Lemma~\ref{lem:straightening}, we assume that $\sigma(vw_i)=\id$ for all $i\in[4]$.
    By the minimality of $G$, we have a $(\vec{D},\sigma)$-packing $\vph$ of $G-\{v,w_1,w_2,w_3,w_4\}$.  
    We now show how to extend $\vph$ to $G$.

    For each $i\in[4]$, let $M_i$ be a matching in $H_v$ (if one exists) such that extending $\vph$ at $v$ by $M_i$ makes it
    impossible to further extend to $w_i$.  When such an $M_i$ exists, 
    since $H_{w_i}-M_i$ is a $(5,2)$-bigraph, by Lemma~\ref{matching-lem}(2) there exists 
    $X_i\subseteq A_i$ (in $H_{w_i}-M_i$) such that $|N_{H_{w_i}-M_i}(X_i)|<|X_i|$.  
    Further, $H_{w_i}-M_i$ is precisely the graph $K_{3,2}+K_{2,3}$.  Thus, for each $i\in[4]$ the set $X_i$ and its incident 
    edges in $M_i$ are uniquely defined; call these edges $\tilde{M_i}$.  (That is, $H_{w_i}-M'_i$ has a 1-factor, for every
    matching $M'_i$ distinct from $M_i$; this is because each edge of $M_i$ joins the two parts of size 3 in $K_{3,2}+K_{2,3}$.)

    By Pigeonhole, there exists a vertex $v_1$ in (at least) all but 
    one such set $X_i$, since $\lceil 4(3)/5\rceil = 3$.  By symmetry, assume $v_1\in X_1\cap X_2\cap X_3$.  Because 
    $\delta(H_v)\ge 4$, we can choose edge $e_1$ incident to $v_1$ but not in $\tilde{M_1}\cup\tilde{M_2}\cup\tilde{M_3}$.  
    We can also choose a vertex $v_4$ in $X_4-v_1$ and an edge $e_4$ incident to $v_4$ but not in $\tilde{M_4}$ that forms a 
    matching with $e_1$.  Note that $H_v-(V(e_1)\cup V(e_4))$ is a $(3,2)$-bigraph, so has a 1-factor by 
    Lemma~\ref{matching-lem}(1).  Thus, $H_v$ has a 1-factor $M_v$ containing
    edges $e_1,e_4$.  We use $M_v$ to extend $\vph$ to $v$; call it $\vph'$.  By construction, we can now extend $\vph'$ to 
    each $w_i$ (since each $\tilde{M_i}$ was unique and none of them is a submatching of $M_v$).
\end{proof}

\section{Planar Graphs with Girth at least 5 and Lists of Size 4}
\label{P5-sec}

In this section, we prove that $\chisc(G)\le 4$ for every planar graph $G$ with girth at least 5.
For the proof, we need the following two easy lemmas.

\begin{lem}\label{canalwaysswapanedge}
    If $H$ is a $(4,2)$-bigraph, then every vertex in $H$ is incident 
    with at least two edges that are each contained in a 1-factor.
\end{lem}

\begin{proof}
    Suppose not. Let $H$ be a $(4,2)$-bigraph and fix $u \in A$ that is incident with at most one edge in a 1-factor. Since $\delta(H) \geq 2$, 
    some edge $uw$ is not contained in a 1-factor. 
    By Hall's Theorem $H - u - w$ contains a set $X \subseteq A$ with $|N(X)| < |X|$. 
    Since $\delta(H) \geq 2$, we know $\delta( H-u-w) \geq 1$. Proposition~\ref{easy-prop} implies that $|X| = 2$ and $|N(X)| = 1$. 
    Thus $H - u - w$ contains vertex-disjoint paths $a_1b_1a_2$ and $b_2a_3b_3$, where $X = \{a_1,a_2\}$, and $a_3 \in A\setminus X$;  
    see Figure~\ref{lem15-fig}.
    Since $\delta(H) = 2$, we know $b_2,b_3\in N_{H}(u)$ and $a_1,a_2\in N_H(w)$. 
    
    Thus, deleting $uw$ from the edges listed above yields a 2-factor, which clearly contains two edge-disjoint 1-factors; namely, 
    $\{ub_2,a_1w, a_2b_1,a_3b_3\}$ and $\{ub_3,a_1b_1,a_2w,a_3b_2\}$. This is
    a contradiction.
    \end{proof}

   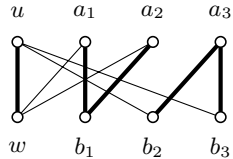
\begin{figure}[!h]
    \centering
\begin{tikzpicture}[xscale = .9, scale=1]
    \def\off{.4cm}
\tikzset{every node/.style=uStyle}

    \foreach \i in {1, 2, 3}
    {
        \draw[thick] (\i,1) node (a\i) {} (\i,0) node (b\i) {};
        \draw (a\i) ++ (0,\off) node[lStyle] {\footnotesize{$a_{\i}$}};
        \draw (b\i) ++ (0,-\off) node[lStyle] {\footnotesize{$b_{\i}$}};
    }

    \draw (0,1) node (u) {} (0,0) node (w) {};
        \draw (u) ++ (0,\off) node[lStyle] {\footnotesize{$u$}};
        \draw (w) ++ (0,-\off) node[lStyle] {\footnotesize{$w$}};

    \draw[ultra thick] (u) -- (w) (a1) -- (b1) -- (a2) (b2) -- (a3) -- (b3);
    \draw (a1) -- (w) -- (a2) (b2) -- (u) -- (b3);
\end{tikzpicture}
\captionsetup{width=.5\linewidth}
       \caption{$H$ contains two edge-disjoint 1-factors. (These appear in the 2-factor formed by deleting edge $uw$.)\label{lem15-fig}}
\end{figure}

\begin{lem}\label{lem:girth5condition}
    If $H$ is a $(4,1)$-bigraph, then $H$ has a 1-factor unless two vertices in a part each have degree 1 and have a common neighbor.
\end{lem}

\begin{proof}
    Suppose not, and let $H$ be a counterexample.  By Hall's Theorem, there exists $X \subseteq A$ such that $|N(X)| < |X|$.  Proposition~\ref{easy-prop} implies that $2\le |X|\le 3$.
    If $|X|=2$, then the two vertices in $X$ have degree 1 and have a common neighbor, a contradiction.
    Similarly, if $|X|=3$ and $|N(X)|\le 
    2$, then two vertices in $B\setminus N(X)$ have degree 1 and have a common
    neighbor, again a contradiction.
\end{proof}

We now prove our main theorem of the section.
By Example~\ref{ex:cycles} we have $\chisc(C_k)=4$ for all $k\ge 3$; thus, the bound in 
Theorem~\ref{thm:upperboundgirth5} is optimal.

\begin{thm}\label{thm:upperboundgirth5}
    If $G$ is a planar graph of girth at least five, then $\chisc(G) \le 4$.  More generally, $\chisc(G)\le 4$ whenever $G$ is 
    triangle-free and has maximum average degree less than $10/3$. 
\end{thm}

\begin{proof}
    By Euler's formula, every planar graph with girth at least $g$ has maximum average degree less than $2g/(g-2)$; so if $g=5$,
    then this average is less than $2(5)/(5-2)=10/3$.  
    Thus, the first statement follows from the second, which we now prove.  Suppose not, and let $(G, (\vec{D},\sigma))$ be a 
    counterexample minimizing $|G|$.  Corollary~\ref{degen-cor} implies that $\delta(G)\ge 3$.  We now show that $G$ contains a 
    3-vertex with at least two adjacent 3-vertices.  Suppose the contrary.  We use discharging where every vertex $v$ has initial 
    charge $d(v)$.  Our single discharging rule is that every 3-vertex takes charge $1/6$ from every neighbor $w$ with $d(w)\ge 4$.  
    Now if $d(v)=3$, then $v$ finishes with at least $3+2(1/6)=10/3$.  And if $d(v)\ge 4$, then $v$ finishes with at least 
    $d(v)-d(v)/6=5d(v)/6\ge 20/6=10/3$.  Thus $G$ has average degree at least $10/3$, which contradicts our hypothesis.  Hence, 
    we assume that $G$ contains a path $xvy$ with $d(x) = d(v) = d(y) = 3$.  Let $z$ be the third neighbor of $v$ distinct 
    from $x$ and $y$. 

    By Lemma~\ref{lem:straightening}, we may assume that $\sigma(e) = \id$ for every edge $e \in \{xv, vy, vz\}$. By minimality, 
    $G-v$ has a $(\vec{D},\sigma)$-packing, $\vph_1, \vph_2, \vph_3, \vph_4$; for short, we call this $\vph$. Let $H_v$ 
    be the auxiliary graph for extending $\vph$ to $v$, where $A = \{1,2,3,4\}$ and $B = \{\vph_1, \vph_2, \vph_3, \vph_4\}$. 
    Since $G$ is a minimum counterexample, $\vph$ does not extend to $v$, and so $H_v$ has no 1-factor. 
    By Lemma~\ref{lem:girth5condition}, $H_v$ contains two vertices of degree 1 with a common neighbor. 

    First suppose $H_v$ has no matching of size three.  Since $\delta(H_v) \geq 1$, there exists a set $X \subseteq A$ 
    with $|X| = 3$ and $|N(X)| = 1$.  By symmetry, we assume that $X = \{1, 2, 3\}$, and that $N(X) = \{\vph_1\}$; see the left of Figure~\ref{case0-fig}.
    This means that in $\vph_1$, all three neighbors of $v$ are colored 4. Unpack $y$, and let $H_y$ be the auxiliary graph for 
    extending $\vph$ to $y$. Note that $\delta(H_y) \geq 2$, so by Lemma~\ref{canalwaysswapanedge} there exists an extension 
    $M_y$ of $\vph$ to $y$ where $\vph_1(y) \neq 4$. After repacking $y$ via $M_y$, the updated auxiliary graph $H_v'$ 
    contains a matching of size three. 
    We assume $H_v'$ has no 1-factor (since otherwise we are done); so, again by Lemma~\ref{lem:girth5condition}, $H_v'$ contains 
    two vertices of degree 1 with a neighbor in common.  By renaming $H_v'$ to $H_v$, we henceforth assume by symmetry that
    $E(H_v) \supset \{1\vph_1, 1\vph_2, 2\vph_3, 3\vph_3, 4\vph_4\}$ and either $d(\vph_1) = d(\vph_2) = 1$ or $d(2) = d(3) = 1$.
    See the right of Figure~\ref{case0-fig}. 

\begin{figure}[!h]
    \centering
\begin{tikzpicture}[xscale = .9, scale=1]
    \def\off{.4cm}
\tikzset{every node/.style=uStyle}

    \begin{scope}
    \foreach \i in {1, 2, 3, 4}
    {
        \draw[thick] (\i,1) node (a\i) {} (\i,0) node (b\i) {};
        \draw (a\i) ++ (0,\off) node[lStyle] {\footnotesize{$\i$}};
        \draw (b\i) ++ (0,-\off) node[lStyle] {\footnotesize{$\vph_{\i}$}};
    }

        \draw (a1) -- (b1) -- (a2) (b1) -- (a3) (b2) -- (a4) -- (b3) (b4) -- (a4);
    \end{scope}

    \begin{scope}[xshift=2in]
    \foreach \i in {1, 2, 3, 4}
    {
        \draw[thick] (\i,1) node (a\i) {} (\i,0) node (b\i) {};
        \draw (a\i) ++ (0,\off) node[lStyle] {\footnotesize{$\i$}};
        \draw (b\i) ++ (0,-\off) node[lStyle] {\footnotesize{$\vph_{\i}$}};
    }
        \draw (b1) -- (a1) -- (b2) (a2) -- (b3) -- (a3) (b4) -- (a4);

    \end{scope}

\end{tikzpicture}
\captionsetup{width=.75\linewidth}
           \caption{Left: The original graph $H_v$.  Right: The new $H_v$, after repacking $y$.
           \label{case0-fig}}
\end{figure}
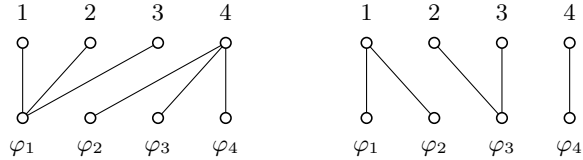

    If $d(\vph_1) = d(\vph_2) = 1$, then the three neighbors $n_1$, $n_2$, and $n_3$ of $v$ satisfy, up to possibly swapping the 
    names of colors or colorings, that $(\vph_1(n_1), \vph_2(n_1)) = (3,2)$ and $(\vph_1(n_2), \vph_2(n_2)) = (4,3)$ and 
    $(\vph_1(n_3), \vph_2(n_3)) = (2,4)$. See the left of Figure~\ref{case1a-fig}.  Since $2\vph_3, 3\vph_3\in E(H_v)$, it follows that $\vph_3(n_3) = \vph_3(n_2) = 1$, 
    and so $\vph_4(n_3) = 3$ and $\vph_4(n_2) = 2$. As $4\vph_4 \in E(H_v)$, we have moreover that $\vph_4(n_1) = 1$ 
    and so $\vph_3(n_1) = 4$. Thus we have
    $\{\vph(x),\vph(y),\vph(z)\}=\{(2,4,1,3),(3,2,4,1),(4,3,1,2)\}$.

    Similarly, if $d(2) = d(3) = 1$, then the three neighbors $n_1$, $n_2$, and $n_3$ of $v$ satisfy, up to possibly swapping the
    names of colors or colorings, that $(\vph_1(n_1),\vph_2(n_1))=(3,2)$ and $(\vph_2(n_2),\vph_4(n_2))=(3,2)$ and 
    $(\vph_4(n_3),\vph_1(n_3))=(3,2)$.  Since $4\vph_4\in E(H_v)$, we have $\vph_4(n_1)=1$ and $\vph_3(n_1)=4$.
    Since $1\vph_1\in E(H_v)$, we have $\vph_1(n_2)=4$ and $\vph_3(n_2)=1$.  Finally, since $1\vph_2\in E(H_v)$, we have
    $\vph_2(n_3)=4$ and $\vph_3(n_3)=1$.
    Thus, in either case $\{\vph(x),\vph(y),\vph(z)\}=\{(2,4,1,3),(3,2,4,1),(4,3,1,2)\}$.
    So $H_v$ is shown on the right in Figure~\ref{case0-fig}.

    We will argue by repacking $x$ and $y$ that there is a $(\vec{D},\sigma)$-packing of $G-v$ that extends to $v$. 
    We split into two cases.  
    As usual, we may write $M_u$ to denote the matching that encodes the packing $\vph$ at $u$.
    When $M_u\cap M_v\ne \emptyset$, we typically write $M_{u\setminus v}$ to denote $M_u\setminus M_v$.
    For brevity, we may also write $(i,j,k,\ell)$, where $\{i,j,k,\ell\}=[4]$, 
    to denote the matching $\{\vph_1i,\vph_2j,\vph_3k,\vph_4\ell\}$.  

    \textbf{Case 1: $\bm{\vph(z) \neq (3,2,4,1)}$. } By symmetry between $x$ and $y$ we assume $\vph(y) = (3,2,4,1)$. 
    (Later, we will consider whether or not $\vph(x)=(2,4,1,3)$.)
    Unpack $y$, and let $H_y$ be the auxiliary graph for extending $\vph$ to $y$.   Note that now $H_v+M_y$ is a 
    $(4,2)$-bigraph that contains the matchings $(1,2,3,4)$ and $(3,1,2,4)$ and $(3,2,4,1)$; see the right of Figure~\ref{case1a-fig}.  

    \begin{figure}[!t]
    \centering
\begin{tikzpicture}[xscale = .9, scale=1]
    \def\off{.4cm}
\tikzset{every node/.style=uStyle}

    \begin{scope}
    \foreach \i in {1, 2, 3, 4}
    {
        \draw[thick] (\i,1) node (a\i) {} (\i,0) node (b\i) {};
        \draw (a\i) ++ (0,\off) node[lStyle] {\footnotesize{$\i$}};
        \draw (b\i) ++ (0,-\off) node[lStyle] {\footnotesize{$\vph_{\i}$}};
    }

    \draw[very thick] (b1) -- (a1) -- (b2) (a2) -- (b3) -- (a3) (b4) -- (a4);
    \draw[-, decorate, decoration={snake, segment length=3pt, amplitude=.5pt}] 
        (b1) -- (a2) (b2) -- (a4) (b3) -- (a1) (b4) -- (a3);
        \draw[dashed] (b1) -- (a4) (b2) -- (a3) (b3)++(0,-.2*\off) --++ (-2cm,1) (b4) -- (a2);
    \draw
        (b1) -- (a3) (b2) -- (a2) (b3) -- (a4) (b4) -- (a1);
    \end{scope}

    \begin{scope}[xshift=2in]
    \foreach \i in {1, 2, 3, 4}
    {
        \draw[thick] (\i,1) node (a\i) {} (\i,0) node (b\i) {};
        \draw (a\i) ++ (0,\off) node[lStyle] {\footnotesize{$\i$}};
        \draw (b\i) ++ (0,-\off) node[lStyle] {\footnotesize{$\vph_{\i}$}};
    }

    \draw[very thick] (b1) -- (a1) (b2) -- (a2) (b3) -- (a3) (b4) -- (a4);
    \draw (b1) -- (a3) (b2) -- (a1) (b3) -- (a2) (b4) ++(-0.20*\off,0) --++ (0,1);
    \draw[dashed] (b1) ++(0,.2*\off) --++ (2,1) (b2) ++(-0.2*\off,0) --++ (0,1) (b3) -- (a4) (b4) -- (a1);
    \end{scope}

\end{tikzpicture}
\captionsetup{width=.75\linewidth}
           \caption{Left: The edges of $H_v$ in bold, together with the edges of $M_{n_1}$ (plain), the edges of $M_{n_2}$ (dashed), and the edges of $M_{n_3}$ (wavy).
           Right: $H_v+M_y$ contains the matchings $(1,2,3,4)$ (bold), $(3,1,2,4)$ (plain), and $(3,2,4,1)$ (dashed).
           \label{case1a-fig}}
\end{figure}
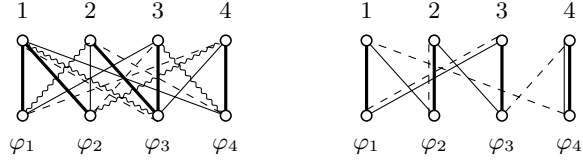

    So if an extension of the packing
    to $y$ cannot be extended to $v$, then it must intersect each of these 3 matchings.
    By Lemma~\ref{canalwaysswapanedge}, since $\delta(H_y) \geq 2$, there exists an extension of $\vph$ to $y$ where 
    $\vph_1(y) \neq 3$; call it $M_y^2$. It is straightforward to check that this extension to $y$ 
    extends (further) to a $(\vec{D},\sigma)$-packing 
    of $G$, using one of the above 3 matchings at $v$, unless $M_y^2=(1,2,3,4)$; to see this, begin by considering how 
    $M_y^2$ could interesect both of the latter two matchings above.
    So we assume $M_y^2=(1,2,3,4)$.  By Lemma~\ref{canalwaysswapanedge}, there exists an extension of $\vph$ to $y$ where 
    $\vph_2(y) \neq 2$; call it $M_y^3$.  By a very similar argument, we can check that this extension extends to $v$, that is, to a 
    $(\vec{D},\sigma)$-packing of $G$ unless $M_y^3=(3,1,2,4)$; so we assume this is the case.
    For consistency, we denote $(3,2,4,1)$ by $M_y^1$.
    So $(3,2,4,1), (1,2,3,4)$, and $(3,1,2,4)$ are all valid extensions of $\vph$ to $y$; 
    we call these the first, second, and third extensions of $\vph$ to $y$, and, as mentioned above, 
    we denote them by $M_y^1$, $M_y^2$, and $M_y^3$.

           \begin{figure}[!h]
    \centering
\begin{tikzpicture}[xscale = .9, scale=1]
    \def\off{.4cm}
\tikzset{every node/.style=uStyle}

    \begin{scope}
    \foreach \i in {1, 2, 3, 4}
    {
        \draw[thick] (\i,1) node (a\i) {} (\i,0) node (b\i) {};
        \draw (a\i) ++ (0,\off) node[lStyle] {\footnotesize{$\i$}};
        \draw (b\i) ++ (0,-\off) node[lStyle] {\footnotesize{$\vph_{\i}$}};
    }

    \draw[very thick] (b1) -- (a1) (b2) -- (a4) (b3) -- (a2) (b4) -- (a3);
    \draw (b1) -- (a2) (b2) -- (a1) (b3) -- (a3) (b4) -- (a4);
    \end{scope}

    \begin{scope}[xshift=2in]
    \foreach \i in {1, 2, 3, 4}
    {
        \draw[thick] (\i,1) node (a\i) {} (\i,0) node (b\i) {};
        \draw (a\i) ++ (0,\off) node[lStyle] {\footnotesize{$\i$}};
        \draw (b\i) ++ (0,-\off) node[lStyle] {\footnotesize{$\vph_{\i}$}};
    }

    \draw[very thick] (b1) -- (a2) (b2) -- (a1) (b3) -- (a4) (b4) -- (a3);
    \draw (b1) -- (a3) (b2) -- (a4) (b3) -- (a2) (b4) -- (a1);
    \end{scope}

    \begin{scope}[xshift=4in]
    \foreach \i in {1, 2, 3, 4}
    {
        \draw[thick] (\i,1) node (a\i) {} (\i,0) node (b\i) {};
        \draw (a\i) ++ (0,\off) node[lStyle] {\footnotesize{$\i$}};
        \draw (b\i) ++ (0,-\off) node[lStyle] {\footnotesize{$\vph_{\i}$}};
    }

    \draw[very thick] (b1) -- (a1) (b2) -- (a2) (b3) -- (a4) (b4) -- (a3);
    \draw (b1) -- (a2) (b2) -- (a4) (b3) -- (a3) (b4) -- (a1);
    \end{scope}

\end{tikzpicture}
\captionsetup{width=.75\linewidth}
           \caption{Left: $H_v^1$ decomposes into the 1-factors $(1,4,2,3)$ and $(2,1,3,4)$.  Center: $H^2_v$ decomposes into the 1-factors $(2,1,4,3)$ and $(3,4,2,1)$.  Right: $H_v^3$ decomposes into the 1-factors $(1,2,4,3)$ and $(2,4,3,1)$.\label{case1b-fig}}
\end{figure}
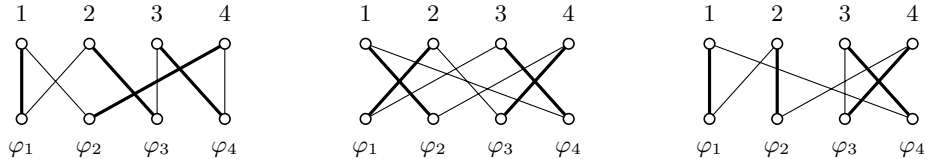

    If $\vph(x)=(2,4,1,3)$ (and $\vph(z)=(4,3,1,2)$), then repack $x$ so that $\vph_1(x) \neq 2$; denote this new packing at $x$
    by $M_x'$.  We first try to repack $y$ with its first extension.  Let $H_v^1:=H_v+M_{x\setminus z}+M_y-M_y^1
    =H_v+M_{x\setminus z}$; see the left of Figure~\ref{case1b-fig}. 
    So we are looking for a 1-factor in $H_v^1$ that does not intersect $M_x'$.  
    In $H_v^1$ we have the disjoint 1-factors $(1,4,2,3)$ and 
    $(2,1,3,4)$.  Thus, $H_v^1-M_x'$ contains a 1-factor, allowing us to extend to $v$, unless
    $M_x'\in \{(1, 2, 3, 4), (1, 3, 2, 4), (1, 4, 3, 2), (3, 1, 2, 4)$, $(4, 1, 2, 3)\}$. 
    Now instead we try repacking $y$ with its second extension.  Let $H_v^2:=H_v+M_y-M_y^2+M_{x\setminus z}$; see the center of  
    Figure~\ref{case1b-fig}.  Similar to above,
    $H_v^2$ has the 1-factors $(2,1,4,3)$ and $(3,4,2,1)$.  So if $H_v^2-M_x'$ has no 1-factor, then $M_x'\in\{(3,1,2,4),(4,1,2,3)\}$.
    Finally, let $H_v^3:=H_v+M_y-M_y^3+M_{x\setminus z}$; see the right of  Figure~\ref{case1b-fig}. 
    For both possibilities for $M_x'$, the graph $H_v^3-M_x'$ contains the 1-factor
    $(2,4,3,1)$.  Thus, we can extend $\vph$ to a $(\vec{D},\sigma)$-packing of $G$.

    If $\vph(x)=(4,3,1,2)$ (and $\vph(z)=(2,4,1,3)$), then repack $x$ so that $\vph_1(x) \neq 4$; denote this new packing at $x$ 
    by $M_x'$.  We first try to repack $y$ with its first extension.  Let $H_v^1:=H_v+M_{x\setminus z}+M_y-M_y^1
    =H_v+M_{x\setminus z}$; see the left of Figure~\ref{case1c-fig}.
    So we are looking for a 1-factor in $H_v^1$ that does not intersect $M_x'$.  In $H_v^1$ we have the 1-factors $(1,3,2,4)$ and 
    $(4,1,3,2)$.  Thus, $H_v^1-M_x'$ contains a 1-factor, allowing us to extend to $v$, unless
    $M_x'\in \{(1, 2, 3, 4), (1, 3, 4, 2), (1, 4, 3, 2), (2, 1, 3, 4),$ $(3, 1, 2, 4)\}$. 
    Now instead we try repacking $y$ with its second extension.  Let $H_v^2:=H_v+M_y-M_y^2+M_{x\setminus z}$; see the center of 
    Figure~\ref{case1c-fig}.  Similar to above,
    $H_v^2$ has the 1-factors $(3,1,4,2)$ and $(4,3,2,1)$.  So if $H_v^2-M_x'$ has no 1-factor, then $M_x'\in\{(1,3,4,2),(3,1,2,4)\}$.
    Finally, let $H_v^3:=H_v+M_y-M_y^3+M_{x\setminus z}$; see the right of Figure~\ref{case1c-fig}.  For both possibilities for $M_x'$, 
    the graph $H_v^3-M_x'$ contains the 1-factor $(4,2,3,1)$.  Thus, we can extend $\vph$ to a $(\vec{D},\sigma)$-packing of $G$.

       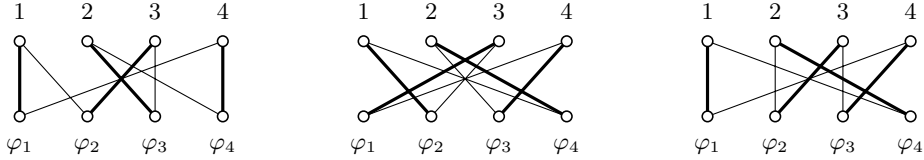
\begin{figure}[!h]
    \centering
\begin{tikzpicture}[xscale = .9, scale=1]
    \def\off{.4cm}
\tikzset{every node/.style=uStyle}

    \begin{scope}
    \foreach \i in {1, 2, 3, 4}
    {
        \draw[thick] (\i,1) node (a\i) {} (\i,0) node (b\i) {};
        \draw (a\i) ++ (0,\off) node[lStyle] {\footnotesize{$\i$}};
        \draw (b\i) ++ (0,-\off) node[lStyle] {\footnotesize{$\vph_{\i}$}};
    }

    \draw[very thick] (a1) -- (b1) (a4) -- (b4) (a2) -- (b3) (a3) -- (b2);
    \draw (b1) -- (a4) (b4) -- (a2) (b3) -- (a3) (b2) -- (a1);
    \end{scope}

    \begin{scope}[xshift=2in]
    \foreach \i in {1, 2, 3, 4}
    {
        \draw[thick] (\i,1) node (a\i) {} (\i,0) node (b\i) {};
        \draw (a\i) ++ (0,\off) node[lStyle] {\footnotesize{$\i$}};
        \draw (b\i) ++ (0,-\off) node[lStyle] {\footnotesize{$\vph_{\i}$}};
    }

    \draw[very thick] (a1) -- (b2) (a3) -- (b1) (a4) -- (b3) (a2) -- (b4);
    \draw (b2) -- (a3) (b1) -- (a4) (b3) -- (a2) (b4) -- (a1);
    \end{scope}

    \begin{scope}[xshift=4in]
    \foreach \i in {1, 2, 3, 4}
    {
        \draw[thick] (\i,1) node (a\i) {} (\i,0) node (b\i) {};
        \draw (a\i) ++ (0,\off) node[lStyle] {\footnotesize{$\i$}};
        \draw (b\i) ++ (0,-\off) node[lStyle] {\footnotesize{$\vph_{\i}$}};
    }

    \draw[very thick] (a1) -- (b1) (a4) -- (b3) (a3) -- (b2) (a2) -- (b4);
    \draw (b1) -- (a4) (b3) -- (a3) (b2) -- (a2) (b4) -- (a1);
    \end{scope}

\end{tikzpicture}
\captionsetup{width=.75\linewidth}
           \caption{Left: $H_v^1$ decomposes into the 1-factors $(1,3,2,4)$ and $(4,1,3,2)$.  Center: $H^2_v$ decomposes into the 1-factors $(3,1,4,2)$ and $(4,3,2,1)$.  Right: $H_v^3$ decomposes into the 1-factors $(1,3,4,2)$ and $(4,2,3,1)$.\label{case1c-fig}}
\end{figure}

\bigskip

    \textbf{Case 2: $\bm{\vph(z) = (3,2,4,1)}$. } By symmetry, we assume that $\vph(y) = (2,4,1,3)$ and $\vph(x) = (4,3,1,2)$.  
    Unpack $y$. By Lemma~\ref{canalwaysswapanedge}, there exists an extension of $\vph$ to $y$ where $\vph_1(y) \neq 2$;
    call this $M_y'$.  Note that $H_v+M_{y\setminus x}$ contains the 1-factors $(1,4,2,3)$ and $(2,1,3,4)$; see the left of 
    Figure~\ref{case2a-fig}.  Thus, $H_v+M_y-M_y'$ contains one of these 1-factors unless 
    $M_y'\in \{(1,2,3,4), (1,3,2,4), (1,4,3,2), (3,1,2,4), (4,1,2,3)\}.$   So we assume this is the case.

         \begin{figure}[!h]
    \centering
\begin{tikzpicture}[xscale = .9, scale=1]
    \def\off{.4cm}
\tikzset{every node/.style=uStyle}

    \begin{scope}
    \foreach \i in {1, 2, 3, 4}
    {
        \draw[thick] (\i,1) node (a\i) {} (\i,0) node (b\i) {};
        \draw (a\i) ++ (0,\off) node[lStyle] {\footnotesize{$\i$}};
        \draw (b\i) ++ (0,-\off) node[lStyle] {\footnotesize{$\vph_{\i}$}};
    }

    \draw[very thick] (b1) -- (a1) (b2) -- (a4) (b3) -- (a2) (b4) -- (a3);
        \draw (b1) -- (a2) (b2) -- (a1) (b3) -- (a3) (b4) -- (a4);
    \end{scope}

    \begin{scope}[xshift=2in]
    \foreach \i in {1, 2, 3, 4}
    {
        \draw[thick] (\i,1) node (a\i) {} (\i,0) node (b\i) {};
        \draw (a\i) ++ (0,\off) node[lStyle] {\footnotesize{$\i$}};
        \draw (b\i) ++ (0,-\off) node[lStyle] {\footnotesize{$\vph_{\i}$}};
    }

    \draw[very thick] (b1) -- (a1) (b2) -- (a3) (b3) -- (a2) (b4) -- (a4);
        \draw (b1) -- (a4) (b2) -- (a1) (b3) -- (a3) (b4) -- (a2);
    \end{scope}

\end{tikzpicture}
\captionsetup{width=.77\linewidth}
           \caption{Left: The graph $H_v+M_y$ contains the 1-factors $(1,4,2,3)$ and $(2,1,3,4)$.\\
           Right: The graph $H_v+M_x$ contains the 1-factors $(1,3,2,4)$ and $(4,1,3,2)$.
           \label{case2a-fig}}
\end{figure}
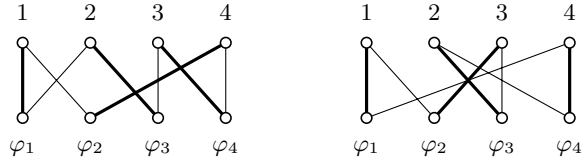

    We now consider $x$: we uncolor $x$, and let $H_x$ be the auxiliary graph for extending $\vph$ to $x$.  By 
    Lemma~\ref{canalwaysswapanedge}, since $\delta(H_x) \geq 2$, there exists an extension of $\vph$ to $x$ where 
    $\vph_1(x) \neq 4$; call it $M_x'$.  Note that $H_v+M_{x\setminus y}$ contains the 1-factors $(1,3,2,4)$ and 
    $(4,1,3,2)$; see the right of Figure~\ref{case2a-fig}.
    Now it is straightforward to check that $H_v+M_{x\setminus y}-M_x'$ contains a 1-factor, 
    allowing us to extend the packing to $G$, unless
    $M_x'\in \{(1,2,3,4), (1,3,4,2), (1,4,3,2), (2,1,3,4), (3,1,2,4)\}$. So we assume this is the case. 

    First suppose $M_y'=(1,3,2,4)$.  Since $M_y=(2,4,1,3)$, it follows 
    that $y$ also admits the packing $(1,4,2,3)$; call it $M_y''$.
    Now $H_v+M_{y\setminus x}-M_y''$ contains the 1-factor $(2,1,3,4)$; see the left of Figure~\ref{case2a-fig}.
    So we assume $M_y'\in \{(1,2,3,4), (1,4,3,2), (3,1,2,4), (4,1,2,3)\}.$

    If $M_y'=(1,2,3,4)$, then we repack $x$ with $M_x'$ and repack $y$ with $M_y'$.  Now $H_v+M_x+M_y-M_y'$ contains 
    the 1-factors $(4,3,1,2)$ and $(4,1,2,3)$; see the left of Figure~\ref{case2b-fig}.
    Thus, after removing $M_x'$, the resulting graph still contains a 1-factor.  
    So we assume $M_y'\in \{(1,4,3,2), (3,1,2,4), (4,1,2,3)\}.$

    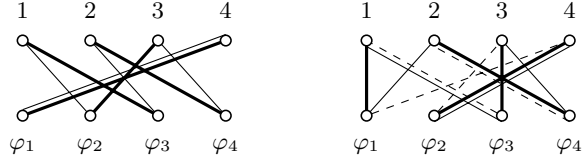
\begin{figure}[!t]
    \centering
\begin{tikzpicture}[xscale = .9, scale=1]
    \def\off{.4cm}
\tikzset{every node/.style=uStyle}

    \begin{scope}
    \foreach \i in {1, 2, 3, 4}
    {
        \draw[thick] (\i,1) node (a\i) {} (\i,0) node (b\i) {};
        \draw (a\i) ++ (0,\off) node[lStyle] {\footnotesize{$\i$}};
        \draw (b\i) ++ (0,-\off) node[lStyle] {\footnotesize{$\vph_{\i}$}};
    }

    \draw[very thick] (b1) -- (a4) (b2) -- (a3) (b3) -- (a1) (b4) -- (a2);
        \draw (b1) ++ (0,.2*\off) --++ (3,1) (b2) -- (a1) (b3) --(a2) (b4) -- (a3);
    \end{scope}

    \begin{scope}[xshift=2in]
    \foreach \i in {1, 2, 3, 4}
    {
        \draw[thick] (\i,1) node (a\i) {} (\i,0) node (b\i) {};
        \draw (a\i) ++ (0,\off) node[lStyle] {\footnotesize{$\i$}};
        \draw (b\i) ++ (0,-\off) node[lStyle] {\footnotesize{$\vph_{\i}$}};
    }

    \draw[very thick] (b1) -- (a1) (b2) -- (a4) (b3) -- (a3) (b4) -- (a2);
        \draw (b1) -- (a2) (b2)++(0,-.2*\off) --++ (2,1) (b3)++(0,-.2*\off) --++ (-2,1) (b4) -- (a3);
        \draw[dashed] (b1) -- (a4) (b2) -- (a3) (b3) -- (a1) (b4)++(0,-.2*\off) --++(-2,1);
    \end{scope}

\end{tikzpicture}
\captionsetup{width=.77\linewidth}
           \caption{Left: When $M_y'=(1,2,3,4)$, the graph $H_v+M_x+M_y-M_y'$ contains the 1-factors $(4,3,1,2)$ and $(4,1,2,3)$.
           Right: When $M_x'=(3,1,2,4)$, the graph $H_v+M_x+M_y-M_x'$ contains the 1-factors $(1,4,3,2)$ and $(2,4,1,3)$ and 
           $(4,3,1,2)$. When $M_y'=(3,1,2,4)$, the graph $H_v+M_x+M_y-M_y'$ is identical.
           \label{case2b-fig}}
\end{figure}

    Instead suppose $M_y'=(1,4,3,2)$, and let $H_v':=H_v+M_y+M_x-M_y'$; note that $H_v'$ contains the 1-factor 
    $(4,1,2,3)$.  Thus, $H_v'-M_x'$ contains a 1-factor unless $M_x'\in \{(2,1,3,4), (3,1,2,4)\}$; so we assume this is the case. 
    Recall that $M_x=(4,3,1,2)$. If $M_x'=(2,1,3,4)$, then $x$ also admits the packing $(2,3,1,4)$; 
    call this $M_x''$.  Note that $H_v'-M_x''$ contains the 1-factor $(4,1,2,3)$, which allows us to extend to $v$.
    Thus we assume $M_x'=(3,1,2,4)$. By Lemma~\ref{canalwaysswapanedge}, since $\delta(H_y) \geq 2$, there exists a 
    repacking $\vph$ of $y$ where $\vph_2(y) \neq 4$; call it $M_y''$. Note that $H_v+M_x+M_y-M_x'$ contains the 1-factors 
    $(1,4,3,2)$ and $(2,4,1,3)$ and $(4,3,1,2)$; see the right of Figure~\ref{case2b-fig}.  
    So $H_v+M_x+M_y-M_x'-M_y''$ must contain one of these 1-factors unless
    $M_y''=(4,3,1,2)$.  
    But now $M_y''=M_x$, so $H_v+M_x+M_y-M_x-M_y''$ contains the 1-factor $(2,1,3,4)$, so we are done.

    Next suppose $M_y'=(3,1,2,4)$. By Lemma~\ref{canalwaysswapanedge}, since $\delta(H_x) \geq 2$, there exists 
    a repacking of $x$ where $\vph_4(x) \neq 2$; call it $M_x''$. 
    Let $H_v':=H_v+M_y-M_y'+M_x$; again, see the right of Figure~\ref{case2b-fig}.  So we are looking for a 1-factor in $H_v'-M_x''$.  Note that $H_v'$ contains the 
    1-factors $(1,4,3,2)$ and $(2,4,1,3)$ and $(4,3,1,2)$.  So it is easy to check that $H_v+M_y-M_y'+M_x-M_x''$ contains a 
    1-factor unless $M_x''=(2,4,1,3)$.  In that case, we use the original packing $M_y$ at $y$, and again we have a 1-factor
    in $H_v+M_{x\setminus y}-M_x''$; namely $(1,3,2,4)$.

Finally, assume $M_y'=(4,1,2,3)$.  Recall $M_x'\in \{(1,2,3,4), (1,3,4,2), (1,4,3,2)$, $(2,1,3,4)$, $(3,1,2,4)\}$. 
If $M_x'\in \{(1,4,3,2),(3,1,2,4)\}$, then $H_v+M_x+M_y-M_x'-M_y'$ contains one of the 1-factors $(1,4,3,2)$ and $(2,3,1,4)$.
If $M_x'=(2,1,3,4)$, then since 
$M_x=(4,3,1,2)$ it follows that $x$ also admits the packing $(4,1,3,2)$; call it $M_x''$. 
In this case, $H_v+M_x+M_y-M_x''-M_y'$ contains the 1-factor $(2,3,1,4)$.
So we assume $M_x'\in\{(1,2,3,4),(1,3,4,2)\}$. 

          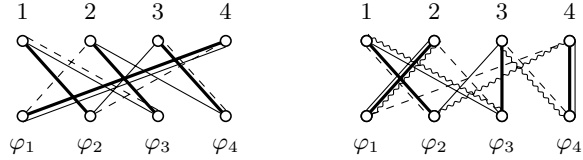
\begin{figure}[!h]
    \centering
\begin{tikzpicture}[xscale = .9, scale=1]
    \def\off{.4cm}
\tikzset{every node/.style=uStyle}

    \begin{scope}
    \foreach \i in {1, 2, 3, 4}
    {
        \draw[thick] (\i,1) node (a\i) {} (\i,0) node (b\i) {};
        \draw (a\i) ++ (0,\off) node[lStyle] {\footnotesize{$\i$}};
        \draw (b\i) ++ (0,-\off) node[lStyle] {\footnotesize{$\vph_{\i}$}};
    }

    \draw[very thick] (b1) -- (a4) (b2) -- (a1) (b3) -- (a2) (b4) -- (a3);
    \draw (b1) ++ (0,-.2*\off) --++ (3,1) (b2) -- (a3) (b3) --(a1) (b4) -- (a2);
         \draw[dashed] (b1) -- (a2) (b2) -- (a4) (b3) ++(0,.2*\off)--++(-2,1) (b4) ++(0,.2*\off) --++ (-1,1);
    \end{scope}

    \begin{scope}[xshift=2in]
    \foreach \i in {1, 2, 3, 4}
    {
        \draw[thick] (\i,1) node (a\i) {} (\i,0) node (b\i) {};
        \draw (a\i) ++ (0,\off) node[lStyle] {\footnotesize{$\i$}};
        \draw (b\i) ++ (0,-\off) node[lStyle] {\footnotesize{$\vph_{\i}$}};
    }

    \draw[very thick] (b1) -- (a2) (b2) -- (a1) (b3) -- (a3) (b4) -- (a4);
        \draw (b1) ++(0,.2*\off) --++(1,1) (b2) -- (a3) (b3) -- (a1) (b4) ++ (.2*\off,0) --++ (0,1);
        \draw[dashed] (b1) -- (a4) (b2)++(0,-.2*\off) --++(-1,1) (b3) -- (a2) (b4) -- (a3);
          \draw[-, decorate, decoration={snake, segment length=3pt, amplitude=.5pt}] 
        (b1) ++  (.2*\off,0) --++ (1,1) (b2) -- (a4) (b3) ++  (0,.15*\off) --++ (-2,1) (b4) ++ (-.2*\off,0) --++ (-1,1);
    \foreach \i in {a1, a2, a3, a4, b1, b2, b3, b4}
    \draw (\i) node {};

    \end{scope}

\end{tikzpicture}
\captionsetup{width=.8\linewidth}
              \caption{Left: When $M_x'=(1,2,3,4)$, the graph $H_v+M_x+M_y-M_x'$ contains the 1-factors $(2,4,1,3)$ and $(4,1,2,3)$ 
              and $(4,3,1,2)$.
           Right: When $M_x'=(1,3,4,2)$, the graph $H_v+M_x+M_y-M_x'$ contains 1-factors $(2,1,3,4)$ and $(2,3,1,4)$ and $(2,4,1,3)$ and 
           $(4,1,2,3)$.
           \label{case2c-fig}}
\end{figure}

Suppose $M_x' = (1,2,3,4)$.
Now $H_v+M_x+M_y-M_x'$ contains the 1-factors $(2,4,1,3)$ and $(4,1,2,3)$ and $(4,3,1,2)$; see the left of Figure~\ref{case2c-fig}.
Since $\delta(H_y)\ge 2$, we can repack $y$ so that $\vph_4(y) \neq 3$; call this $M_y''$.
Thus, $H_v+M_x+M_y-M_x'-M_y''$ contains a 1-factor unless $M_y''=(4,3,1,2)$.
In this case, $H_v+M_{y\setminus x}-M_y''$ contains the 1-factor $(1,4,2,3)$.
Now instead assume $M_x' = (1,3,4,2)$. Again, we consider the repacking $M_y''$ at $y$ where $\vph_4(y) \neq 3$. 
The graph $H_v + M_x - M_x'+ M_y$ contains the 1-factors $(2,1,3,4)$ and $(2,3,1,4)$ and $(4,1,2,3)$; see the right of Figure~\ref{case2c-fig}.
Thus, after removing $M_y''$ this graph still contains a 1-factor unless $M_y'' = (2,1,3,4)$. 
In this case, $H_v + M_{y\setminus x} - M_y''$ contains the 1-factor (1,4,2,3).
\end{proof}

[Note added in proof: The hypothesis of being triangle-free was 
removed in~\cite[Theorem 4.1]{CCvBZ}, using that paper's Lemma~3.1, which handles the case 
of 3-vertices inducing a $P_3$ or $C_3$.  The proof of that Lemma~3.1 is brute force, handled by a computer.  
Similarly, we can use the $C_3$ case of that Lemma~3.1 to remove the triangle-free hypothesis of Theorem~\ref{thm:upperboundgirth5}.]

\section{Planar Graphs with Lists of Size 8}
\label{P3-sec}
In this section we prove that $\chisc(G)\le 8$ for every planar graph $G$.  We will need a number of lemmas on 
matchings in $(8,3)$-bigraphs and $(8,4)$-bigraphs.  To focus on the proof of our main result, we defer most of these lemmas
to the next section.  However, we do begin this section with a key definition and an easy proposition regarding that definition.

\begin{defn}
\label{type-defn}
    In an $(8,3)$-bigraph $G$, an \emph{obstruction} to a 1-factor is a set $X\subseteq A$ with $|N(X)|<|X|$.  We will focus 
    specifically on 4 types of obstructions.  An obstruction $X\subseteq A$ 
    \begin{itemize}
        \item has \emph{type 1} if $|X|=5$ and $|N(X)|=3$; 
        \item has \emph{type 2} if $|X|=4$ and $|N(X)|=3$, but there exists $x_1\in A\setminus X$ and edge $e_1$ 
            incident with $x_1$ such that $|N_G(X\cup\{x_1\})|=4$ but $|N_{G-\{e_1\}}(X\cup \{x_1\})|=3$; 
        \item has \emph{type 3} if $|X|=4$ and $|N(X)|=3$, but there exists $x_1\in A\setminus X$ and edges $e_1,e_2$ 
            incident with $x_1$ such that $|N_G(X\cup\{x_1\})|=5$ but $N_{G-\{e_1,e_2\}}(X\cup \{x_1\})|=3$; 
        \item and has \emph{type 4} if $|X|=4$ and $|N(X)|=3$ but $X$ is not a subset of any obstruction of an earlier type.
    \end{itemize}
\end{defn}

\begin{figure}[!h]
    \centering
    \begin{tikzpicture}[semithick, scale=.95, yscale=1.2, scale=.8]
\tikzset{every node/.style=uStyle}

\begin{scope}
    \foreach \i in {1,...,8}
        \draw[ultra thick] (\i,1) node (x\i) {} (\i,-.07) node (y\i) {};
    \draw (x1) node[ugStyle] {} (x2) node[ugStyle] {} (x3) node[ugStyle] {} (x4) node[ugStyle] {} (x5) node[ugStyle] {};

    \foreach \i in {1,2,3,4,5}
    \foreach \j in {1,2,3}
        \draw (x\i) -- (y\j);

    \foreach \i in {6,7,8}
    \foreach \j in {4,5,6,7,8}
        \draw (x\i) -- (y\j);
\end{scope}

\begin{scope}[xshift = 3.75in]
    \foreach \i in {1,...,8}
        \draw[ultra thick] (\i,1) node (x\i) {} (\i,-.07) node (y\i) {};
    \draw (x1) node[ugStyle] {} (x2) node[ugStyle] {} (x3) node[ugStyle] {} (x4) node[ugStyle] {} (x5) node[ubStyle] {};

    \foreach \i in {1,2,3,4}
    \foreach \j in {1,2,3}
        \draw (x\i) -- (y\j);
    \draw (x5) -- (y1) (x5) -- (y2) (x5) -- (y4);

    \foreach \i in {6,7,8}
    \foreach \j in {5,6,7,8}
        \draw (x\i) -- (y\j);
    \draw (y4) -- (x6) (y4) -- (x8);
\end{scope}

\begin{scope}[xshift = 3.75in, yshift = -1.0in]
    \foreach \i in {1,...,8}
        \draw[ultra thick] (\i,1) node (x\i) {} (\i,-.07) node (y\i) {};
    \draw (x1) node[ugStyle] {} (x2) node[ugStyle] {} (x3) node[ugStyle] {} (x4) node[ugStyle] {} (x5) node[ubStyle] {};

    \foreach \i in {1,2,3,4}
    \foreach \j in {1,2,3}
        \draw (x\i) -- (y\j);
    \draw (x5) -- (y1) (x5) -- (y4) (x5) -- (y5);

    \foreach \i in {6,7,8}
   \foreach \j in {4,5,6,7,8}
        \draw (x\i) -- (y\j);
\end{scope}

\begin{scope}[yshift = -1.0in]
    \foreach \i in {1,...,8}
        \draw[ultra thick] (\i,1) node (x\i) {} (\i,-.07) node (y\i) {};
    \draw (x1) node[ugStyle] {} (x2) node[ugStyle] {} (x3) node[ugStyle] {} (x4) node[ugStyle] {};

    \foreach \i in {1,2,3,4}
    \foreach \j in {1,2,3}
        \draw (x\i) -- (y\j);

    \foreach \i in {5,6,7,8}
    \foreach \j in {4,5,6,7,8}
        \draw (x\i) -- (y\j);
\end{scope}

\end{tikzpicture}
    \captionsetup{width=.7\textwidth}
    \caption{Clockwise from top left: Example obstructions having types 1, 2, 3, and 4.\label{type-fig}
    Vertices in $X$ are gray and vertex $x_1$, if it exists, is black.}
\end{figure}
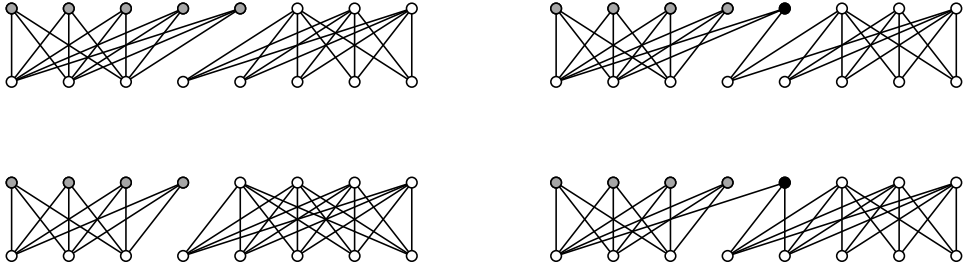

It is straightforward to prove the following.  If $H$ is an $(8,3)$-bigraph with an obstruction $X$ having type 1, 2, 3, or 4, then 
$H$ contains exactly one such obstruction; furthermore, if $X$ has type 2, then $x_1$ and $e_1$ are uniquely defined, and if $X$ has 
type 3, then $x_1$, $e_1$, $e_2$ are all uniquely defined.  But this is not needed for our main result, and its proof 
is a bit tedious, so we omit the details.

For an $(s,t)$-bigraph $H$, let \emph{$\swap(H)$} denote the same bigraph $H$, but with the names of the parts $A$ and $B$ swapped.

\begin{prop}
    \label{type-prop}
    If $H$ is an $(8,3)$-bigraph with no 1-factor, then either (a) $H$ contains an obstruction having type 1 or 2 
    and $\swap(H)$ also contains an obstruction of the same type 
    or (b) at least one of $H$ and $\swap(H)$ contains an obstruction having type 3 or 4.
\end{prop}
\begin{proof}
    By Hall's Theorem, we assume $H$ has an obstruction $X$; choose $X$ to maximize $|X|$.  
    By Proposition~\ref{easy-prop}, we know $|X|\in\{4,5\}$.
    If $|X|=5$ and $|N(X)|=3$, then $X$ has type 1.  If $|X|=4$, then by definition $X$ has type 2, 3, or 4.
    Instead suppose that $|X|=5$ and $|N(X)|=4$, but $X$ contains no obstruction having type 2.  
    Now we instead consider $\swap(G)$.  Formally, let $\tilde{X}:=B\setminus N(X)$.  Note that
    $|\tilde{X}|=4$ and $|N(\tilde{X})|=|A\setminus X|=3$.   Since $X$ contains no obstruction having type 2, it is easy to check 
    that $\tilde{X}$ does not have type 2; so $\tilde{X}$ has type 3 or type 4 in $\swap(G)$.
    Finally, it is easy to check that if $G$ has an obstruction having type 1 or 2, then so does $\swap(G)$.
\end{proof}

We need the following structural lemma of Borodin~\cite{borodin}.  For a short proof sketch, also see~\cite[Theorem~3.4]{guide}.

\begin{lem}\cite{borodin}
    \label{borodin-lem}
    If $G$ is a planar graph with $\delta(G)\ge 5$, then $G$ contains a 3-cycle $uvw$ such that $d(u)+d(v)+d(w)\le 17$.
\end{lem}

Now we prove the main result of this section.

\begin{thm}
    If $G$ is a planar graph, then $\chisc(G)\le 8$.
\end{thm}
\begin{proof}
    Our proof is by reducibility, with the unavoidability (typically done via discharging) handled by Lemma~\ref{borodin-lem}.
    We assume the theorem is false and pick a counterexample $(G,(\vec{D},\sigma))$ that minimizes $|G|$; here $(\vec{D},\sigma)$ is an 
    8-cover.
    By Corollary~\ref{degen-cor}, we know $\delta(G)\ge 5$.  By Lemma~\ref{borodin-lem}, we know $G$ contains a 
    triangle $uvw$ such that $d(u)+d(v)+d(w)\le 17$.  We first handle the case that two vertices in $\{u,v,w\}$, 
    say $v$ and $w$, are 5-vertices; the bulk of the proof focuses on the other case: $d(u)=d(v)=6$ and $d(w)=5$.

    In each case, we start with a $(\vec{D},\sigma)$-packing $\vph$ of $G-w$, by minimality, and aim to extend $\vph$ to $w$.  
    If we cannot, then we unpack $v$, and repack $v$ to facilitate extending to $w$.  The second case is harder, 
    since we may need to unpack both $u$ and $v$, and so must consider modifying $H_w$ by adding and removing 
    \emph{two} matchings each, rather than just one.

    \textbf{Case 1: $\bm{d(v)=d(w)=5}$.}  By Lemma~\ref{lem:straightening}, we assume $\sigma(vw)=\id$.  
    If we cannot extend $\vph$ to $w$, then Proposition~\ref{type-prop} implies that auxiliary graph $H_w$ 
    (or $\swap(H_w)$; but, by symmetry, we assume it is in $H_w$) 
    has an obstruction $X_w$ having type 1, 2, 3, or 4.  By Lemma~\ref{switcher-general-lem}, it suffices to 
    repack $v$ so that the updated auxiliary graph $H'_w$ for $w$ has a
    matching of the appropriate size, 1 or 2, from $B\setminus N_{H_w}(X_w)$ to $X_w\cup\{x_1\}$ (or simply to $X_w$). 

    Note that each $x\in X_w$ has $d_{H_w}(x)=3$, because $\delta(H_w)\ge 3$ and $d_{H_w}(x)\le |N_{H_w}(X_w)|=3$ 
    by the definition of type 1, 2, 3 and 4 obstructions.  
    Because $d_{H_w}(x)=3=8-5$ for all $x\in X_w$, and $d_G(w)=5$, every neighbor of $w$ in $G$ must be 
    responsible for forbidding a distinct edge of $K_{8,8}$ incident to $x$ in $H_w$.  
    In particular, the matching $M_v$ encoding the packing at $v$ must contain a matching $\tilde{M_v}$ from $X_w$ to 
    $B_w\setminus N_{H_w}(X_w)$  that saturates $X_w$.  (This observation is central to our argument and we repeat it frequently.)
    For clarity, we remark that $M_v$ and $\tilde{M_v}$ are not contained in $H_w$, but rather in its ``bipartite complement''
    $K_{8,8}-E(H_w)$.

    Let $J$ consist of an arbitrary 2 edges in $\tilde{M_v}$.
    Now uncolor $v$ and note that $H_v$ is an $(8,4)$-bigraph.  Let $H'_v:=H_v-J$.  By Corollary~\ref{matching-inc-cor}(2), the graph 
    $H'_v$ has a 1-factor $M'_v$.  Form $\vph'$ from $\vph$ by repacking $v$ according to $M'_v$.  
    Let $H'_w$ denote the auxiliary graph to extend $\vph'$ to $w$.  
    As noted above, the edges of $J$ are forbidden in $H_w$ only by the packing at $v$.  Thus, when we repack $v$, 
    disallowing the edges of $J$ in the corresponding 1-factor $M'_v$ of $H_v$, the edges of $J$ are no longer forbidden 
    for packing $w$.  That is, $H'_w$ contains $J$.
    Thus, Lemma~\ref{switcher-general-lem} implies that $H'_w$ has a 1-factor $M'_w$.  
    Now we use $M'_w$ to extend $\vph'$ to a $(\vec{D},\sigma)$-packing of $G$.

    \textbf{Case 2: $\bm{d(u)=d(v)=6}$ and $\bm{d(w)=5}$.}
    By Lemma~\ref{lem:straightening}, we assume $\sigma(vw)=\id$ and $\sigma(uw)=\id$ (but possibly $\sigma(uv)\ne\id$).  
    Let $H_w$ be the auxiliary graph to extend $\vph$ to $w$.  Note that $H_w$ is an $(8,3)$-bigraph since $d(w)=5$.  
    If $H_w$ has a 1-factor, then we can use it to extend $\vph$ to $w$ and are done.
        Thus, by Proposition~\ref{type-prop}, and symmetry between $A$ and $B$, we assume that $H_w$ contains an obstruction 
    $X_w$ having type 1, 2, 3, or 4.

    Suppose that $X_w$ has type 4.  Now each $x\in X_w$ has $d_{H_w}(x)=3$.
    Let $M_v$ be the 1-factor (in $H_v$ if we were to unpack $v$) denoting $\vph$ restricted to $v$. So $M_v$ contains 
    a matching $\tilde{M_v}$
    from $X_w$ to $B_w\setminus N_{H_w}(X_w)$ that saturates $X_w$.  By Lemma~\ref{switcher-simple-lem}, there exists $e\in \tilde{M_v}$ 
    such that $H_v-e$ contains another 1-factor $M_v'$.  Now let $H_w':=H_w+M_v-M_v'$ and note that
    $H_w'$ is again an $(8,3)$-bigraph.  Furthermore, $H_w'$ contains $e$ from $X_w$ to $B_w\setminus N_{H_w}(X_w)$, since $e$ was
    previously forbidden by $M_v$, but is not forbidden by $M_v'$.  Thus, by 
    Lemma~\ref{switcher-general-lem}(4), we know $H_w'$ contains a 1-factor $M_w'$.  Now starting from $\vph$, we unpack $v$, extend 
    the $(\vec{D},\sigma)$-packing to $v$ by $M_v'$, and further extend it to $w$ by $M_w'$.  Thus, we assume $H_w$ has no 
    obstruction having type 4.

    Suppose instead that $X_w$ has type 3.  
    As above, let $M_v$ be the 1-factor (in $H_v$ if we were to unpack $v$) denoting $\vph$ restricted to $v$. 
    Note that each $y\in B_w\setminus N_{H_w}(X_w\cup \{x_1\})$ has 
    $d_{H_w}(y)=3$.  So $M_v$ contains a matching $\tilde{M}_v$ from $X_w\cup\{x_1\}$ to $B_w\setminus N_{H_w}(X_w\cup\{x_1\})$ 
    that saturates $B_w\setminus N_{H_w}(X_w\cup\{x_1\})$.  If no edge of $\tilde{M}_v$ is incident to $x_1$, then 
    Lemma~\ref{switcher-simple-lem} implies that $\tilde{M}_v$ contains an edge $e$ such that $H_v-e$ contains another 
    1-factor $M_v'$.  Now let $H_w':=H_w+M_v-M_v'$ and note that $H_w'$ is again an $(8,3)$-bigraph.  Furthermore, $H_w'$ 
    contains $e$ from $X_w$ to $B_w\setminus N_{H_w}(X_w)$.  And $e$ combines in $H'_w$ with some edge $e'$ incident to $x_1$ to form a matching of size 2 from $X_w\cup\{x_1\}$ to $N_{H_w}(X_w)$.  Thus, by Lemma~\ref{switcher-general-lem}(3), we know $H_w'$ 
    contains a 1-factor $M_w$, and we are done as above.  
    
    So instead assume that $\tilde{M}_v$ contains one edge $e'$ 
    incident to $x_1$ and each other edge $e\in \tilde{M}_v$ is such that $H_v-e$ has no 1-factor.  
    Recall that $d_{H_w}(x)=3$ for all $x\in X_w$, since $X_w$ has type 3.  So $d_{H_w}(x)=8-5=|L(w)|-d_G(w)$ implies that $M_v$ saturates $X_w$.
    Now we instead take $e$ to be some edge 
    of $M_v$ from $X_w$ to $N_{H_w}(x_1)\setminus N_{H_w}(X_w)$ (such an edge exists because $M_v$ saturates $X_w$), and get $M_v'$ by 
    Lemma~\ref{switcher-simple-lem}.  Again we let $H_w':=H_w+M_v-M_v'$. Clearly, $e\in H_w'$, since $e$ was forbidden 
    from $H_w$ only by $M_v$ and is not forbidden by $M'_v$.  But also, since $e'\in M_v$,
    we get that $|N_{H_w+M_v}(x_1)\setminus N_{H_w}(X)|=3$, so $|N_{H'_w}(x_1)\setminus N_{H_w}(X)|\ge 2$.  
    Thus, $e$ combines with some edge incident to $x_1$ in $H'_w$ to form a matching of size 2.
    That is, $H_w'$ contains a matching of size 2 from $X_w\cup\{x_1\}$ to $B_w\setminus N_{H_w}(X_w)$.
    Hence, by Lemma~\ref{switcher-general-lem}(3), we know $H_w'$ contains a 1-factor $M_w$, so we are done as above.

    Now suppose that $X_w$ has type 1 or type 2.  We first handle the case that $X_w$ has type 1, and later comment about 
    how to adapt the proof to the case when $X$ has type 2.  Let $M_v$ and $M_u$ be the 1-factors denoting $\vph$ restricted 
    to $v$ and $u$, respectively.  Since $X_w$ has type 1, both $M_v$ and $M_u$ contain matchings of size 5 from $X_w$ to 
    $B_w\setminus N_{H_w}(X_w)$.  Since $d_{H_w}(x) = 3$ for all $x \in X_w$ and $d_{H_w}(y) = 3$ for all $y \in B_w\setminus N_{H_w}(X_w)$, these two matchings are disjoint.
    Let $C$ denote the union of these two matchings, and note that $C$ is a 2-regular graph.  
    Either $C$ consists of a vertex disjoint 4-cycle and 6-cycle or else $C$ consists of a 10-cycle.  In each case let $\tilde{C}$ 
    consist of a path of length 5 in $C$.  Now unpack $v$ and $u$.  Note that the resulting $H_u$ is an $(8,4)$-bigraph.  
    Let $H_u':=H_u-\tilde{C}$.   By Corollary~\ref{matching-inc-cor}(2), we know that $H_u'$ has a 1-factor, which we can use to 
    extend the $(\vec{D},\sigma)$-packing to $u$; call this new $(\vec{D},\sigma)$-packing $\vph'$.  If we can extend $\vph'$ to $v$, 
    then we can further extend (this extension to $v$) to $w$. 
    To see this, note that for any matching $M'_v$, the subgraph $\tilde{C}-M'_v$ contains a matching of size 2.
    Thus, the new auxiliary graph $H'_w$, has a matching of size 2 from $X_w$ to $B_w\setminus N_{H_w}(X_w)$.
    So, by Lemma~\ref{switcher-double-lem}, with $\tE=\emptyset$ and $k=4$, we can extend this $(\vec{D},\sigma)$-packing to $w$.

    (Now we adapt to the case when $X_w$ has type 2. Note that $d_{H_w}(x)=3$ for all $x\in X_w$ and $d_{H_w}(x_1)\le 4$.
    So both $M_v$ and $M_u$ saturate $X_w$ and at least one saturates $x_1$.  
    If $M_v$ and $M_u$ each contain matchings of size 5 into $X_w\cup\{x_1\}$ analogous to the case above, 
    then we take $C$ as before; otherwise, one contains a matching of size 5 and the other a matching of size 4, 
    and to get our 2-regular 10-vertex subgraph $C$ (with vertex set $X_w\cup\{x_1\}\cup (B\setminus N_{H_w}(X_w))$) 
    we add to these matchings the edge $e_1$.  Now everything else works nearly the same, 
    but we must apply Lemma~\ref{switcher-double-lem}, to $X_w\cup\{x_1\}$, with $|\tE|=1$.) 
    So we conclude that $\vph'$ does not extend to~$v$.

    Consider the auxiliary graph $H'_v$ for extending $\vph'$ to $v$.  Since $\vph'$ does not extend, we know that $H'_v$ has
    an obstruction; call it $X_v$.  By Proposition~\ref{type-prop}, we assume 
    that $X_v$ has type 1, 2, 3, or 4.  We will now repack $u$ so that we can extend the resulting $(\vec{D},\sigma)$-packing 
    $\vph'$ to both $v$ and $w$.  To extend $\vph'$ to $v$, we will use Lemma~\ref{switcher-general-lem}, and to extend this 
    extension to $w$ we will use Lemma~\ref{switcher-double-lem}.  So what remains is to find a good 1-factor in $H_u$, after 
    we unpack $u$ and $v$, that allows these extensions.

    Suppose that $X_v$ has type 4. Since $d_{H_v}(x)=3$ for all $x\in X_v$, and $d_{G-w}(v)=5$, we know that $M_u$
    contains a matching $\tilde{M}_u$ that forbids in $H_v$ a matching between $X_v$ and $B_v\setminus N_{H_v}(X_v)$ that 
    saturates $X_v$; possibly $\tilde{M}_u$ is distinct from the matching that it forbids in $H_v$, because possibly 
    $\sigma(uv)\ne\id$.  Form $M$ from 
    $\tilde{M}_u$ by adding an arbitrary (vertex disjoint) edge $e'$ with one endpoint in each part of $H_u$.  Recall the
    definition of $C$ from above, either a 10-cycle or a 6-cycle and 4-cycle.  
    By Lemmas~\ref{key1factor-lem} and \ref{key1factorB-lem}, we know that
    $H_u$ contains a 1-factor $M'_u$ that avoids $2P_3$ from $C$ and avoids $2K_2$ from $M$.  One edge in this $2K_2$ 
    might be $e'$, but the other forbids in $H_v$ an edge $e''$ between $X_v$ and $N_{H_v}(X_v)$.  
    Now $e''\in H'_v$, so by Lemma~\ref{switcher-general-lem} we can extend $\vph'$ to $v$, via some 1-factor $M'_v$ in $H'_v$.  
    Finally, since $M'_u$ avoids a $2P_3$ from $C$, we know that $M'_u\cup M'_v$ avoid a $2K_2$ in $C$.
    Thus, we can extend (this extension to $v$) to $w$ by Lemma~\ref{switcher-double-lem}.

    Suppose instead that $X_v$ has type 3.  The argument is nearly the same, but we must choose $J$ more carefully, since we 
    now need a matching of size 2 between $X_v$ and $B_v\setminus N_{H_v}(X_v)$ after we repack $u$.  
    For 3 vertices $y \in B_v\setminus N_{H_v}(X_v)$, we have $d_{H_v}(y)=3$.
    If an edge from $x_1$ to any of these 3 vertices is forbidden in $H_v$ by some edge of $M_u$, then the argument above 
    (for type 4) still works.  In that case, after $u$ is repacked, $x_1$ will still have at least two edges into 
    $B_v\setminus N_{H_v}(X_v)$, so a single
    additional edge $e'$ from $X_v$ to $B_v\setminus N_{H_v}(X_v)$ will combine with some edge $e''$ incident to $x_1$ to give 
    the desired matching of size 2 from $X_v\cup\{x_1\}$ to $B_v\setminus N_{H_v}(X_v)$.  So we assume that no such edge 
    from $x_1$ to one of these 3 vertices is forbidden by an edge of $M_u$.
    Note that the matching $M_u$ must contain a matching $\tilde{M}_u$ that forbids in $H_v$ a matching $\widehat{M_u}$ 
    between $X_v$ and these 3 specified vertices.  Further, $M_u$ contains an edge that forbids another edge $e'$ from $X_v$
    into $B_v\setminus N_{H_v}(X_v)$ from the vertex of $X_v$ not saturated by $\widehat{M_u}$.
    Let $M:=\widehat{M_u}\cup\{e',e''\}$, where $e''$ is from $x_1$ to the final vertex of $B_v\setminus N_{H_v}(X_v)$.
    We repeat the argument in the previous paragraph with this choice of $M$.

    Suppose instead that $X_v$ has type 1 or type 2.  If $X_v$ has type 1, then $M_u$ contains a matching $\tilde{M}_u$ that forbids
    a matching in $H_v$ from
    $X_v$ to $B_v\setminus N_{H_v}(X_v)$ that saturates $X_v$.  Now we repeat the argument from above with $M:=\tilde{M}_u$.  Suppose
    instead that $X_v$ has type 2 and that $M_u$ does not contain such a matching $\tilde{M}_u$ of size 5 that forbids in $H_v$
    a matching of size 5 from $X_v\cup \{x_1\}$ to $B_v\setminus N_{H_v}(X_v)$.  In this case,
    $M_u$ does contain such a matching $\tilde{M}_u$ of size 4 that forbids in $H_v$ a matching $\widehat{M_u}$ 
    of size 4 from $X_v$ to $B_v\setminus N_{H_v}(X_v)$.  Further, $\widehat{M_u}\cup\{e_1^u\}$ is the desired matching of size 5; 
    here let $e_1$ be as in the definition of a type 2 obstruction in $H_v$, and let $e_1^u$ be the edge in $M_u$ 
    (if it exists) that would forbid $e_1$ in $M_v$.

    This completes the proof.
\end{proof}

\section{Lemmas on Matchings: (8,3)- and (8,4)-Bigraphs}
\label{matching-sec}
In this section, we prove many of the lemmas needed in the previous section for the proof of Theorem~\ref{P3-thm}.

\begin{lem}
    \label{matching-inc-lem}
    Let $H$ be a bigraph with $|A|\in\{2k,2k+1\}$, for some positive integer $k$.
    Denote the vertices in $A$ and $B$ by, respectively, $a_1,\ldots,a_{|A|}$ and $b_1,\ldots,b_{|A|}$ where $d(a_i)\le 
    d(a_{i+1})$ and $d(b_i)\le d(b_{i+1})$ for all $i\in[|A|-1]$.  Suppose $d(a_i)\ge\min\{i-1,k\}$ and 
    $d(b_i)\ge\min\{i-1,|B|-k\}$ for all $i\in[|A|]$. 
    \begin{enumerate}
        \item[(1)] Now $H$ has a 1-factor unless there exists $i'\in[k]$ such that $|N(\{a_1,\ldots,a_{i'}\})|=i'-1$ or
        there exists $i'\in[|B|-k]$ such that $|N(\{b_1,\ldots,b_{i'}\})|=i'-1$.
        \item[(2)] In particular, $H$ has a 1-factor if $d(a_i)\ge i$ and $d(b_i)\ge i$ for all $i\in[k]$.
    \end{enumerate}
\end{lem}
\begin{proof}
    Note that (2) follows immediately from (1), so we prove (1).
    By Hall's Theorem, it suffices to show that $|N(X)|\ge |X|$ for all $X\subseteq A$ 
    unless there exists some exceptional set $\{a_1,\ldots,a_{i'}\}$ or $\{b_1,\ldots,b_{i'}\}$ as in (1).
    Let $i:=|X|$.  First suppose that $i\le k$.
    If $X\ne\{a_1,\ldots,a_i\}$, then $X$ contains $a_{\ell}$ for some $\ell>i$, so $|N(X)|\ge d(a_{\ell})\ge \min\{\ell-1,k\}\ge i 
    = |X|$, and we are done.  If instead $X=\{a_1,\ldots,a_i\}$, then either $|N(X)|=i-1$ or $|N(X)|\ge i$; but in each case we are
    done.  So suppose instead that $i\ge k+1$.  For all $j\ge |A|-i+2$, by Pigeonhole we have 
    $|N(b_j)\cap X|\ge |X|+((|A|-i+2)-1) - |A| = 1$.  Thus, $b_j\in N(X)$.  
    If there exists $b_{j'}$ such that $j'\le |A|-i+1$ and $b_{j'}\in N(X)$, 
    then $|N(X)| \ge |B| - ((|B|-i+2)-1)+1 =|B|-|B|+i-2+1+1 = i = |X|$.  
    Otherwise, $|N(\{b_1,\ldots,b_{|A|-i+1}\})|=|A|-i$; and again we are done.
\end{proof}

\begin{cor}
    \label{matching-inc-cor}
    Let $H$ be a bigraph where $|A|=|B|=8$ and $a_1,\ldots,a_8$ and $b_1,\ldots,b_8$ are the 
    vertices in parts $A$ and $B$ with $d(a_i)\le d(a_{i+1})$ and $d(b_i)\le d(b_{i+1})$ for all $i\in[7]$.
    \begin{enumerate}
        \item[(1)] Suppose $d(a_1)\ge 1$, $d(a_2)\ge 2$, $d(a_3)\ge 3$, $d(a_4)\ge 3$, $d(a_i)\ge 4$ for all $i\in\{5,6,7,8\}$ and
            $d(b_1)\ge 1$, $d(b_2)\ge 2$, $d(b_3)\ge 3$, $d(b_4)\ge 3$, $d(b_j)\ge 4$ for all $j\in\{5,6,7,8\}$.  
            Now $H$ has a 1-factor unless $|N(\{a_1,\ldots,a_4\})|=3$ or $|N(\{b_1,\ldots,b_4\})|=3$.
        \item[(2)] In particular, $H$ has a 1-factor if $d(a_i)\ge i$ and $d(b_i)\ge i$ for all $i\in[4]$ and $d(a_i)\ge 4$ 
            and $d(b_i)\ge 4$ for all $i\in\{5,\ldots,8\}$.
    \end{enumerate}
    \begin{proof}
        This follows immediately from part (1) of the previous lemma, with $k=4$.
    \end{proof}
\end{cor}

For what follows, it may be helpful to recall Definition~\ref{type-defn} and the examples in Figure~\ref{type-fig}.
\begin{lem}
    \label{switcher-general-lem}
    Let $H$ be an $(8,3)$-bigraph with an obstruction X having type $i$, for some $i\in[4]$.  Let $H'$ be another $(8,3)$-bigraph, 
    where $H':=H+M_1-M_2$ and $M_1$ and $M_2$ are matchings.  Now $H'$ has a 1-factor if any of the following 4 conditions hold:
    \begin{enumerate}
        \item[(1)] $X$ has type 1 and $H'$ contains a matching of size 2 from $X$ to $B\setminus N_H(X)$; or
        \item[(2)] $X$ has type 2 and $H'$ contains a matching of size 2 from $X\cup\{x_1\}$ to $B\setminus N_H(X)$; or
        \item[(3)] $X$ has type 3 and $H'$ contains a matching of size 2 from $X\cup\{x_1\}$ to $B\setminus N_H(X)$; or
        \item[(4)] $X$ has type 4 and $H'$ contains an edge from $X$ to $B\setminus N_H(X)$.
    \end{enumerate}
\end{lem}
\begin{proof}
    The proofs for all types are similar, so we present them together.
    By Hall's Theorem, it suffices to show that $|N_{H'}(X')|\ge |X'|$ for all $X'\subseteq A$.  By Proposition~\ref{easy-prop},
    we may assume $|X'|\in\{4,5\}$.  For each type, if $X'\subseteq A\setminus X$, then $|X'|=4$ and $|N_{H'}(X')|=
    |N_H(A\setminus X)|=5 > |X'|$.  So we assume that $X'\cap X\ne\emptyset$.

    Suppose $X'\subseteq X$ (if $X$ has type 1 or 4) or $X'\subseteq X\cup\{x_1\}$ (if $X$ has type 2 or 3).  Now $|X'\cap X|\ge 3$
    so $|N_{H'}(X')\cap N_H(X)|=|N_H(X)|=3$.  And $H'$ contains a matching from $X'$ to $B\setminus N_H(X)$ of size $|X'|-3$.
    Thus, $|N_{H'}(X')|\ge |N_H(X)|+(|X'|-3) = |X'|$, as desired.  So there must exist $x_2,x_3\in X'$ with $x_2\in X$ and 
    $x_3\in A\setminus X$.  In fact, if $X$ has type 2 or 3, then we can assume $x_3\in A\setminus (X\cup\{x_1\})$.
    Thus $|N_{H'}(X')|\ge |N_{H'}(x_2)\cap N_H(X)|+|N_{H'}(x_3)\cap (B\setminus N_H(X))|\ge 2+2$.  If $|X'|=4$, then we are done.
    Now assume $|X'|=5$.  If there exist $x_2',x_2''\in X'\cap X$, then $|N_{H'}(\{x_2',x_2''\})\cap N_H(X)|=3$ and we are done 
    (since $H[\{x_2',x_2''\}\cup N_H(\{x_2',x_2''\})]=K_{2,3}$).
    So instead we must have $x_3',x_3'',x_3'''\in X'\setminus X$.  But now $|N_{H'}(\{x_3',x_3'',x_3'''\})|=|B\setminus N_H(X)|=5$, and again we 
    are done.
\end{proof}

\begin{lem}
    \label{switcher-simple-lem}
    Let $H$ be an $(8,3)$-bigraph.  If $M$ is a 1-factor in $H$, then there exists $M'\subseteq M$ with $|M'|\ge 6$
    such that for each $e\in M'$, the graph $H-e$ contains a 1-factor.
\end{lem}
\begin{proof}

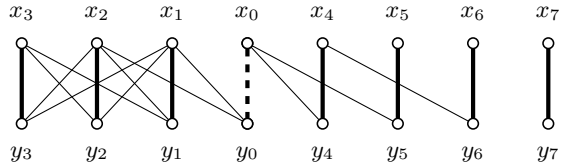
\begin{figure}[!b]
    \centering
\begin{tikzpicture}[yscale=1]
    \def\off{.4cm}
\tikzset{every node/.style=uStyle}

    \foreach \i/\x in {0/0, 1/-1, 2/-2, 3/-3, 4/1, 5/2, 6/3, 7/4}
    {
        \draw[ultra thick] (\x,1) node (x\i) {} -- (\x,-.07) node (y\i) {};
        \draw (x\i) ++ (0,\off) node[lStyle] {\footnotesize{$x_{\i}$}};
        \draw (y\i) ++ (0,-\off) node[lStyle] {\footnotesize{$y_{\i}$}};
    }
    \foreach \i/\j in {1/0, 1/1, 1/2, 1/3, 2/0, 2/1, 2/2, 2/3, 3/1, 3/2, 3/3, 0/4, 0/5, 4/6}
        \draw (x\i) -- (y\j);
    \draw[line width = .7mm, white] (x0) -- (y0);
    \draw (x0) node {} (y0) node {};
    \draw[ultra thick, dashed] (x0) -- (y0);

\end{tikzpicture}
    \caption{A subset of the edges present in $H-x_0y_0$; here the 1-factor $M$ is drawn in bold.\label{switcher-simple-fig}}
\end{figure}

    Fix $H$ and $M$ as in the lemma.  Suppose there exists $e\in M$ such that $H-e$ has no 1-factor, and denote $e$ by $x_0y_0$,
    where $x_0\in A$ and $y_0\in B$.  Let $x_1,x_2$ be two other neighbors of $y_0$, and denote the edges of $M$ incident to $x_1$
    and $x_2$ by $x_1y_1$ and $x_2y_2$; see Figure~\ref{switcher-simple-fig}.  
    Note that $y_1x_0,y_2x_0\notin E(H)$; otherwise $M\setminus\{x_0y_0,x_iy_i\}\cup\{x_0y_i,
    x_iy_0\}$ is a 1-factor in $H-e$, a contradiction.  So there exists $x_3\in N(y_1)\setminus\{x_0,x_1,x_2\}$.  Let $x_3y_3$
    denote the edge of $M_3$ incident to $x_3$.  
    Note that $y_3x_0 \not \in E(H)$, as otherwise $M \setminus 
    \{x_0y_0, x_1y_1, x_3y_3\} \cup \{x_0y_3, x_1y_0, x_3y_1\}$ is a 1-factor in $H-e$, a contradiction. 
    Repeating the same argument from $x_0$, we find vertices $y_4, y_5\in N(x_0)$ and
    edges $x_4y_4$, $x_5y_5$, $x_6y_6$, all in $M$, with $y_6\in N(x_4)\setminus\{y_0,y_4,y_5\}$.  Note that these edges are 
    distinct from $x_iy_i$ for all $i\in\{1,2,3\}$; otherwise $H$ has an $M$-alternating cycle $C$, similar to above, and the symmetric difference
    $M\triangle C$ is a 1-factor in $H-e$.

    Denote the final edge of $M$ by $x_7y_7$.  If there exists $i\in\{0,1,2,3\}$ and $j\in\{0,4,5,6\}$ with $\{y_ix_7,x_jy_7\} \subset E(H)$, then
    again $H$ contains an $M$-alternating cycle $C$ and we are done, since $M\triangle C$ is a 1-factor in $H-e$, which contradicts
    our choice of $e$.
    So assume by symmetry that no such $y_ix_7$ exists.  Similarly, we can assume there exists no edge $y_ix_j$ for all 
    $i\in\{0,1,2,3\}$ and $j\in\{0,4,5,6\}$.  So, since $\delta(H)\ge 3$, we have $x_iy_j\in E(H)$ for all $i,j\in\{1,2,3\}$.  
    Thus, we can include $x_iy_i$ in $M'$ for all $i\in\{1,2,3\}$.  We must prove the analogous statement for at least 3 values of
    $i$ in $\{4,5,6,7\}$. 
    Since $\delta(H)\ge 3$, each $x_i$ with $i\in\{4,5,6,7\}$ has at least two neighbors in $\{y_4,y_5,y_6,y_7\}\setminus\{y_i\}$.
    So, by Pigeonhole (8 edges into 6 pairs), there exist at least two pairs $i',i''\in\{4,5,6,7\}$ such that $x_{i'}y_{i''},x_{i''}y_{i'}\in E(H)$.
    Now we are done, since for each $i$ in the union of these pairs, we can include $x_iy_i\in M'$.
\end{proof}

\begin{lem}
    \label{switcher-double-lem}
    Fix an integer $k\ge 4$.
    Let $H$ be a $(2k,k-1)$-bigraph such that there exist $X\subseteq A$ and $\tE\subset E(H)$ with $|\tE|\le 1$ such
    that $|X|-|N_{H-\tE}(X)|=2$.  Let $H'$ be a $(2k,k-1)$-bigraph, where $H'$ is formed from $H$ by adding two matchings and 
    removing two matchings.  If $H'$ contains a matching of size 2 from $X$ to $B\setminus N_{H-\tE}(X)$, then $H'$ has a 1-factor.  
\end{lem}
    Our main interest in this lemma is when $k=4$, in which case $X$ is an obstruction having type 1 or type 2.
    However, the proof of this more general form is no harder, so we include it.
\begin{proof}
    By Hall's Theorem, it suffices to show $|N_{H'}(X')| \ge |X'|$ for all $X'\subseteq A$.  By Proposition~\ref{easy-prop},
    since $H'$ is a $(2k,k-1)$-bigraph, this is true whenever $|X'|\notin\{k,k+1\}$.  
    If $|X|=k$, then $X$ contains some vertex $x$ not incident with any edge in $\tE$.
    Since $\delta(H)\ge k-1$, we have $|N_{H-\tE}(X)|\ge d_H(x)\ge k-1$; thus,
    $|X|-|N_{H-\tE}(X)|\le k-(k-1)\le 1$, a contradiction.
    Hence, we must have $|X|=k+1$ and $|N_{H-\tE}(X)|=k-1$.  
    Now $H\supseteq (K_{k+1,k-1}-e)+(K_{k-1,k+1}-e)$ with at most one edge between the subsets of size $k+1$ in $A$ and $B$.
    Consider $X'\subseteq A$.  

    First suppose that $|X'|=k$.  
    If $|X'\setminus X|\ge 3$, then $|N_{H'}(X')|\ge |B\setminus N_{H-\tE}(X)|-1\ge k=|X'|$. 
    (To see that no two vertices in $B\setminus N_{H-\tE}(X)$ can be excluded from $N_{H'}(X')$, note that the 5 edges of $K_{3,2}-e$ cannot be covered by the 2 matchings removed from $H$ when forming $H'$.)  
    If $|X'\setminus X|=2$, then $|X'\cap X|\ge 2$.  So $|N_{H'}(X')|\ge |N_{H'}(X'\setminus X)\cap (B\setminus N_{H-\tE}(X))|+
    |N_{H'}(X'\cap X)\cap N_{H-\tE}(X))| \ge (|B\setminus N_{H-\tE}(X)|-2)+(|N_{H-\tE}(X)|-2) \ge k-1+k-3= 2k-4 \ge k$, since $k\ge 4$. 
    If $|X'\setminus X|=1$, then $|X'\cap X|\ge 3$.  So $|N_{H'}(X')|\ge |N_{H'}(X'\setminus X)\cap (B\setminus N_{H-\tE}(X))|+
    |N_{H'}(X'\cap X)\cap N_{H-\tE}(X))| \ge (|B\setminus N_{H-\tE}(X)|-3)+(|N_{H-\tE}(X)|-1) \ge k-2+k-2= 2k-4 \ge k$, since $k\ge 4$. 
    If $|X'\setminus X|=0$, then $|X'\cap X| = k$.  So $N_{H-\tE}(X)\subseteq N_{H'}(X')$, but also $H'$ has a matching of 
    size 2 from $X$ to $B\setminus N_{H-\tE}(X)$, and at least one of its edges is incident to a vertex of $X'\cap X$.
    Thus, $|N_{H'}(X')| \ge (k-1)+1 = k$.

    Suppose instead that $|X'|=k+1$.  
    If $|X'\setminus X|\ge 4$, then $|N_{H'}(X')|\ge |B\setminus N_{H-\tE}(X)|\ge k+1=|X'|$. 
    If $|X'\setminus X|=3$, then $|X'\cap X|\ge 2$.  So $|N_{H'}(X')|\ge |N_{H'}(X'\setminus X)\cap (B\setminus N_{H-\tE}(X))|+
    |N_{H'}(X'\cap X)\cap N_{H-\tE}(X))| \ge (|B\setminus N_{H-\tE}(X)|-1)+(|N_{H-\tE}(X)|-2) \ge k+k-3= 2k-3 \ge k+1$, 
    since $k\ge 4$. 
    If $|X'\setminus X|=2$, then $|X'\cap X|\ge 3$.  So $|N_{H'}(X')|\ge |N_{H'}(X'\setminus X)\cap (B\setminus N_{H-\tE}(X))|+
    |N_{H'}(X'\cap X)\cap N_{H-\tE}(X))| \ge (|B\setminus N_{H-\tE}(X)|-2)+(|N_{H-\tE}(X)|-1) \ge k-1+k-2= 2k-3 \ge k+1$, 
    since $k\ge 4$. 
    If $|X'\setminus X|=0$, then $|X'\cap X| = k+1$, i.e. $X'=X$.  So $N_{H-\tE}(X)\subseteq N_{H'}(X')$, but also $H'$ 
    has a matching of size 2 from $X'$ to $B\setminus N_{H-\tE}(X)$.  Thus $|N_{H'}(X')| \ge (k-1)+2 = k+1$.
\end{proof}

\begin{lem}
    Let $H$ be an $(8,4)$-bigraph with a 10-cycle $C$ and also with a matching $M$ of size 5 (necessarily sharing vertices with $C$).  
    There exists a subgraph $J$ of $H$ that contains a $2P_3$ from $C$ and a $2K_2$ from $M$ such that $H-E(J)$ has a 1-factor.
    \label{key1factor-lem}
\end{lem}

\begin{proof}
    Our general approach will be to guarantee a 1-factor in $H-E(J)$ by Corollary~\ref{matching-inc-cor}.  
    Throughout, we denote $H-E(J)$ by $H'$.
    We want to ensure that $|N_{H'}(\{a_1,\ldots,a_4\})|\ge 4$ and $|N_{H'}(\{b_1,\ldots,b_4\})|\ge 4$. So we will often consider 
    a few options for $J$ and choose an option such that both inequalities hold.
    Throughout, we denote $C$ by $x_1y_1\cdots x_5y_5$.  Note that by Pigeonhole, $M$ has at least two edges with 
    endpoints in $A\cap V(C)$, and at least two edges (possibly the same two edges, but not necessarily) with endpoints 
    in $B\cap V(C)$.  For each edge $e$ of $M$ with both endpoints on $C$, the \emph{length} of $e$, denoted $\ell(e)$, 
    is the distance between its endpoints measured along $C$.  So if an edge's length is defined, then it is 1, 3, or 5.
    (We only use edge length to motivate the distinction between Cases 1 and 2.)

    If $M$ contains at least two edges of $C$, i.e. two edges $e$ with $\ell(e)=1$, then we simply \emph{form $J$ by choosing} our $2P_3$ to 
    contain both of these edges, with the endpoints of the two copies of $P_3$ in opposite parts. By Corollary~\ref{matching-inc-cor}(2), $H-E(J)$ has a 1-factor. We split into cases for the remainder of the analysis: if $M$ contains an
    edge $e$ with length equal to 3 or 5, then we are in Case 1.  Otherwise, at most one edge of $M$ has length well-defined.
    So there exist edges ${a_1b_1}$ and ${a_2b_2}$ such that ${a_1,b_2\in V(C)}$ but 
    ${b_1,a_2\notin V(C)}$, where ${a_1,a_2\in A}$ and ${b_1,b_2\in B}$.  Thus, we are in Case 2.

    \textbf{Case 1: $\bm{M}$ contains edge $\bm{y_1x_i}$ for some $\bm{i\in\{3,4\}}$.} Let $E(J):=\{x_1y_1,y_1x_2,y_1x_i,x_iy_{i-1},$
    $x_iy_i,ab\}$, where we will carefully choose, as we explain shortly, $ab\in M-y_1x_i$, with $a\in A$ and $b\in B$; possibly here $\{a,b\}\cap V(C)\ne \emptyset$.  
    We give details here only in the case where 
    no edge of $M$ lies on $C$ and our choices of $a$ and $b$ satisfy
    $a \not \in \{x_1,x_2\}$ and $b \not \in \{y_{i-1}, y_i\}$; 
    the other cases are similar but easier.
    The left of Figure~\ref{key1factor-fig1} shows the case when $i=3$. 
    We have $d_{H'}(y_1)\ge 1$, 
    $d_{H'}(y_{i-1})\ge 3$, $d_{H'}(y_i)\ge 3$, $d_{H'}(b)\ge 3$, and $d_{H'}(x_1)\ge 3$, $d_{H'}(x_2)\ge 3$, 
    $d_{H'}(x_i)\ge 1$, $d_{H'}(a)\ge 3$.  So $H'$ has a 1-factor by Corollary~\ref{matching-inc-cor} as long as both 
    $|N_{H'}(\{y_1,y_{i-1},y_i,b\})|\ge 4$ and $|N_{H'}(\{x_1,x_2,x_i,a\})|\ge 4$.  
    First suppose that $d_H(y_5)=4$ and $d_H(x_{i+1})=4$.  If we have
    $|N_{H'}(\{y_1,y_{i-1},y_i,b\})|= 3$, then $y_{i-1}x_{i+1}\in E(H)$.   But now $x_{i+1}$ has exactly one neighbor other than 
    $y_1, y_{i-1},y_i$.  So $N_{H'}(b)\ne N_{H'}(y_i)$ for all but at most one choice of $ab \in M-y_1x_i$.  Similarly, if
    $|N_{H'}(\{x_1,x_2,x_i,a\})|= 3$, then $x_2y_5\in E(H)$.  But now $y_5$ has exactly one neighbor other than $x_1,x_2,x_5$.
    So $N_{H'}(a)\ne N_{H'}(x_1)$ for all but at most one choice of $ab \in M-y_1x_i$.  We have 4 choices for $ab$ and at most 
    two are forbidden, so some choice of $ab$ works.

\begin{figure}[!h]
    \centering
\begin{tikzpicture}[xscale = .75, scale=.8, yscale=-1]
    \def\off{.4cm}
\tikzset{every node/.style=uStyle}
\tikzstyle{gStyle}=[shape = circle, minimum size = 4pt, inner sep = 1pt,
outer sep = 0pt, fill=gray!50!white, semithick, draw]

\begin{scope}[xshift = -3.15in]
    \foreach \i/\x in {5/1, 1/2, 2/3, 3/4, 4/5}
    {
        \draw[thick] (\x,0) node (x\i) {} (\x+.5,1) node (y\i) {};
        \draw (x\i) ++ (0,-\off) node[lStyle] {\footnotesize{$x_{\i}$}};
        \draw (y\i) ++ (0,\off) node[lStyle] {\footnotesize{$y_{\i}$}};
    }

    \draw[ultra thick] (6,0) node (a) {} -- (6.5,1) node (b) {};
    \draw (a) ++ (0,-\off) node[lStyle] {\footnotesize{$a$}};
    \draw (b) ++ (0,\off) node[lStyle] {\footnotesize{$b$}};

    \draw (x5) -- (y5) -- (x1) -- (y1) -- (x2) -- (y2) -- (x3) -- (y3) -- (x4) -- (y4) -- (x5);
    \draw[ultra thick] (x1) -- (y1) -- (x2) (y1) -- (x3) (y2) -- (x3) -- (y3);
\end{scope}

\begin{scope}[xshift = 0in]
    \foreach \i/\x in {5/1, 1/2, 2/3, 3/4, 4/5}
    {
        \draw[thick] (\x,0) node (x\i) {} (\x+.5,1) node (y\i) {};
        \draw (x\i) ++ (0,-\off) node[lStyle] {\footnotesize{$x_{\i}$}};
        \draw (y\i) ++ (0,\off) node[lStyle] {\footnotesize{$y_{\i}$}};
    }

    \draw[ultra thick] (6,0) node (a) {} -- (6.5,1) node (b) {};
    \draw (a) ++ (0,-\off) node[lStyle] {\footnotesize{$a$}};
    \draw (b) ++ (0,\off) node[lStyle] {\footnotesize{$b$}};

    \draw (x5) -- (y5) -- (x1) -- (y1) -- (x2) -- (y2) -- (x3) -- (y3) -- (x4) -- (y4) -- (x5);
    \draw[ultra thick] (y5) -- (x1) -- (y1) -- (x3) (y2) -- (x3) -- (y3);

    \draw (y5) node[gStyle] {};
\end{scope}

\begin{scope}[xshift = 3.15in]
    \foreach \i/\x in {5/1, 1/2, 2/3, 3/4, 4/5}
    {
        \draw[thick] (\x,0) node (x\i) {} (\x+.5,1) node (y\i) {};
        \draw (x\i) ++ (0,-\off) node[lStyle] {\footnotesize{$x_{\i}$}};
        \draw (y\i) ++ (0,\off) node[lStyle] {\footnotesize{$y_{\i}$}};
    }

    \draw[ultra thick] (6,0) node (a) {} -- (6.5,1) node (b) {};
    \draw (a) ++ (0,-\off) node[lStyle] {\footnotesize{$a$}};
    \draw (b) ++ (0,\off) node[lStyle] {\footnotesize{$b$}};

    \draw (x5) -- (y5) -- (x1) -- (y1) -- (x2) -- (y2) -- (x3) -- (y3) -- (x4) -- (y4) -- (x5);
    \draw[ultra thick] (y5) -- (x1) -- (y1) -- (x3) (x3) -- (y3) -- (x4);

    \draw (y5) node[gStyle] {} (x4) node[gStyle] {};
\end{scope}

\end{tikzpicture}
    \caption{Possible choices of $J$ in Case 1; vertices with degree known to be at least 5 are
    gray.\label{key1factor-fig1}}
\end{figure}
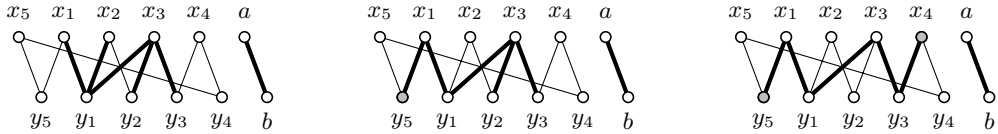

    Now assume (by symmetry) $d_H(y_5)\ge 5$.  If $d_H(x_{i+1})=4$, then let $E(J):=\{y_5x_1,x_1y_1,y_1x_i,$
    $x_iy_{i-1},x_iy_i,ab\}$; see the center of Figure~\ref{key1factor-fig1}.  
    Again, at most one choice of $ab$ is bad, but we have at least 4 such choices.  
    Finally, assume $d_H(x_{i+1})\ge 5$; see the right of Figure~\ref{key1factor-fig1}. Now let 
    $E(H):=\{y_5x_1,x_1y_1,y_1x_i,x_iy_i,y_ix_{i+1},ab\}$, for an arbitrary edge $ab\in M-y_1x_i$.  
    Now at most 3 vertices in each part of $H'$ have degree at most 3.  So $H'$ has a 1-factor.
    This completes Case~1.

    \textbf{Case 2: $\bm{M}$ contains edges $\bm{a_1b_1}$ and $\bm{a_2b_2}$ such that $\bm{a_1,b_2\in V(C)}$ but 
    $\bm{b_1,a_2\notin V(C)}$, where $\bm{a_1,a_2\in A}$ and $\bm{b_1,b_2\in B}$.}
    By symmetry, we assume $a_1=x_1$ and $b_2\in\{y_1,y_2,y_3\}$.

    First suppose that $b_2=y_1$.  
    Let $E(J):=\{x_5y_5,y_5x_1,x_1b_1,a_2y_1,y_1x_2,x_2y_2\}$.
    We must verify that $|N_{H'}(\{x_1,x_2,x_5,a_2\})|\ge 4$ and that $|N_{H'}(\{y_1,y_2,y_5,b_1\})|\ge 4$.  
    The first inequality holds because $y_1\in N_{H'}(x_1)\setminus N_{H'}(a_2)$,
    and the second holds because $x_1\in N_{H'}(y_1)\setminus N_{H'}(b_1)$.

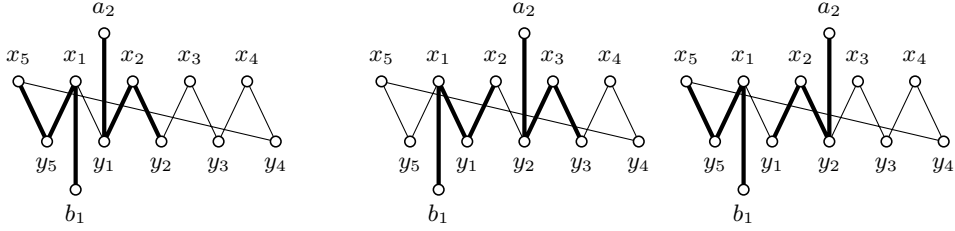
\begin{figure}[!h]
    \centering
\begin{tikzpicture}[xscale = .95, scale=.8, yscale=-1]
    \def\off{.4cm}
\tikzset{every node/.style=uStyle}
\tikzstyle{gStyle}=[shape = circle, minimum size = 4pt, inner sep = 1pt,
outer sep = 0pt, fill=gray!50!white, semithick, draw]

\begin{scope}[xshift = -2.5in]
    \foreach \i/\x in {5/1, 1/2, 2/3, 3/4, 4/5}
    {
        \draw[thick] (\x,0) node (x\i) {} (\x+.5,1) node (y\i) {};
        \draw (x\i) ++ (0,-\off) node[lStyle] {\footnotesize{$x_{\i}$}};
        \draw (y\i) ++ (0,\off) node[lStyle] {\footnotesize{$y_{\i}$}};
    }

    \draw (2,1.8) node (b){} (2.5,-.8) node (a) {};
    \draw (b) ++ (0,\off) node[lStyle] {\footnotesize{$b_1$}};
    \draw (a) ++ (0,-\off) node[lStyle] {\footnotesize{$a_2$}};
    \draw (x5) -- (y5) -- (x1) -- (y1) -- (x2) -- (y2) -- (x3) -- (y3) -- (x4) -- (y4) -- (x5);

    \draw[ultra thick] (x5) -- (y5) -- (x1) -- (b)  (a) -- (y1) -- (x2) -- (y2);
\end{scope}

\begin{scope}[xshift = 0in]
    \foreach \i/\x in {5/1, 1/2, 2/3, 3/4, 4/5}
    {
        \draw[thick] (\x,0) node (x\i) {} (\x+.5,1) node (y\i) {};
        \draw (x\i) ++ (0,-\off) node[lStyle] {\footnotesize{$x_{\i}$}};
        \draw (y\i) ++ (0,\off) node[lStyle] {\footnotesize{$y_{\i}$}};
    }

    \draw (2,1.8) node (b){} (3.5,-.8) node (a) {};
    \draw (b) ++ (0,\off) node[lStyle] {\footnotesize{$b_1$}};
    \draw (a) ++ (0,-\off) node[lStyle] {\footnotesize{$a_2$}};
    \draw (x5) -- (y5) -- (x1) -- (y1) -- (x2) -- (y2) -- (x3) -- (y3) -- (x4) -- (y4) -- (x5);

    \draw[ultra thick] (b) -- (x1) -- (y1) -- (x2) (a) -- (y2) -- (x3) -- (y3);
\end{scope}

\begin{scope}[xshift = 2.1in]
    \foreach \i/\x in {5/1, 1/2, 2/3, 3/4, 4/5}
    {
        \draw[thick] (\x,0) node (x\i) {} (\x+.5,1) node (y\i) {};
        \draw (x\i) ++ (0,-\off) node[lStyle] {\footnotesize{$x_{\i}$}};
        \draw (y\i) ++ (0,\off) node[lStyle] {\footnotesize{$y_{\i}$}};
    }

    \draw (2,1.8) node (b){} (3.5,-.8) node (a) {};
    \draw (b) ++ (0,\off) node[lStyle] {\footnotesize{$b_1$}};
    \draw (a) ++ (0,-\off) node[lStyle] {\footnotesize{$a_2$}};
    \draw (x5) -- (y5) -- (x1) -- (y1) -- (x2) -- (y2) -- (x3) -- (y3) -- (x4) -- (y4) -- (x5);

    \draw[ultra thick] (x5) -- (y5) -- (x1) -- (b)  (y1) -- (x2) -- (y2) -- (a);
\end{scope}

\end{tikzpicture}
    \caption{Left: The choice of $J$ in Case~2 when $b_2=y_1$. 
    Center and right: The choices of $J$ in Case~2 when $b_2=y_2$.
    \label{key1factor-fig2}}
\end{figure}

    Next suppose that $b_2=y_2$. Let $E(J):=\{x_1b_1,x_1y_1,y_1x_2,y_2a_2,y_2x_3,x_3y_3\}$.
    Now $d_{H'}(x_1)\ge 2$, $d_{H'}(x_2)\ge 3$, 
    $d_{H'}(a_2)\ge 3$  $d_{H'}(x_3)\ge 2$, 
    and $d_{H'}(b_1)\ge 3$, $d_{H'}(y_1)\ge 2$, $d_{H'}(y_2)\ge 2$, $d_{H'}(y_3)\ge 3$.  
    Note that $y_2\in N_{H'}(x_2)\setminus N_{H'}(a_2)$.  Thus, $|N_{H'}(\{x_1,x_2,x_3,a_2\})|\ge 4$.
    So if $J$ fails, then it is because $|N_{H'}(\{b_1,y_1,y_2,y_3\})|=3$.  This requires that $b_1x_2,b_1x_4,y_3x_2\in E(H)$.
    Now instead let $E(J):=\{x_5y_5,y_5x_1,x_1b_1,y_1x_2,x_2y_2,y_2a_2\}$; see the right of Figure~\ref{key1factor-fig2}.  Note that $x_1\in N_{H'}(y_1)\setminus N_{H'}(b_1)$,
    so $|N_{H'}(\{y_5,b_1,y_1,y_2\})|\ge 4$.  Thus, if $J$ fails, then it is because $|N_{H'}(\{x_5,x_1,x_2,a_2\})|=3$.  Since
    $b_1x_2\in E(H)$, this requires $b_1a_2\in E(H)$; similarly, it requires $y_1,y_2,y_3,y_4\in N_H(a_2)$.  But now $d_H(a_2)\ge 5$,
    a contradiction since then $|N_{H'}(a_2)| \geq 4$.

    Finally, suppose that $b_2=y_3$.  Let $E(J):=\{x_1y_5,x_1b_1,x_1y_1, y_3x_3, y_3a_2, y_3x_4\}$; see the 
    left of Figure~\ref{key1factor-fig3}.  We assume that this $J$ fails.
    By symmetry between $A$ and $B$, we assume that $|N_{H'}(\{y_5,b_1,y_1,y_3\})|=3$.  This requires that $d_H(b_1)=4$ and 
    $x_5,x_2\in N_H(b_1)$.  By symmetry, we assume that $x_4\notin N_H(b_1)$.  Now instead let $E(J):=\{x_1b_1,x_1y_1,y_1x_2, 
    y_2x_3,x_3y_3,y_3a_2\}$; see the center of Figure~\ref{key1factor-fig3}.  Since $x_4\in N_{H'}(y_3)\setminus N_{H'}(b_1)$, 
    we have $|N_{H'}(\{b_1,y_1,y_2,y_3\})|\ge 4$.
    So we assume that $|N_{H'}(\{x_1,x_2,x_3,a_2\})|=3$.  This requires that $d_H(a_2)=4$ and $y_2,y_3,y_5\in N_H(a_2)$.
    But now we modify $J$ by removing $x_1y_1,y_1x_2$ and adding $x_1y_5,y_5x_5$; see the right of Figure~\ref{key1factor-fig3}.  
    The same arguments now require that $y_4,y_1\in N_H(a_2)$.  But now $d_H(a_2)\ge 5$, a contradiction.
\end{proof}

    \begin{figure}[!h]
    \centering
\begin{tikzpicture}[xscale = .95, scale=.8, yscale=-1]
    \def\off{.4cm}
\tikzset{every node/.style=uStyle}
\tikzstyle{gStyle}=[shape = circle, minimum size = 4pt, inner sep = 1pt,
outer sep = 0pt, fill=gray!50!white, semithick, draw]

\begin{scope}[xshift = -2.5in]
    \foreach \i/\x in {5/1, 1/2, 2/3, 3/4, 4/5}
    {
        \draw[thick] (\x,0) node (x\i) {} (\x+.5,1) node (y\i) {};
        \draw (x\i) ++ (0,-\off) node[lStyle] {\footnotesize{$x_{\i}$}};
        \draw (y\i) ++ (0,\off) node[lStyle] {\footnotesize{$y_{\i}$}};
    }

    \draw (2,1.8) node (b){} (4.5,-.8) node (a) {};
    \draw (b) ++ (0,\off) node[lStyle] {\footnotesize{$b_1$}};
    \draw (a) ++ (0,-\off) node[lStyle] {\footnotesize{$a_2$}};
    \draw (x5) -- (y5) -- (x1) -- (y1) -- (x2) -- (y2) -- (x3) -- (y3) -- (x4) -- (y4) -- (x5);

    \draw[ultra thick] (y5) -- (x1) (b) -- (x1) -- (y1) (x3) -- (y3) -- (a) (y3) -- (x4);
\end{scope}

\begin{scope}[xshift = 0in]
    \foreach \i/\x in {5/1, 1/2, 2/3, 3/4, 4/5}
    {
        \draw[thick] (\x,0) node (x\i) {} (\x+.5,1) node (y\i) {};
        \draw (x\i) ++ (0,-\off) node[lStyle] {\footnotesize{$x_{\i}$}};
        \draw (y\i) ++ (0,\off) node[lStyle] {\footnotesize{$y_{\i}$}};
    }

    \draw (2,1.8) node (b){} (4.5,-.8) node (a) {};
    \draw (b) ++ (0,\off) node[lStyle] {\footnotesize{$b_1$}};
    \draw (a) ++ (0,-\off) node[lStyle] {\footnotesize{$a_2$}};
    \draw (x5) -- (y5) -- (x1) -- (y1) -- (x2) -- (y2) -- (x3) -- (y3) -- (x4) -- (y4) -- (x5);

    \draw[ultra thick] (b) -- (x1) -- (y1) -- (x2) (y2) -- (x3) -- (y3) -- (a);
\end{scope}

\begin{scope}[xshift = 2.5in]
    \foreach \i/\x in {5/1, 1/2, 2/3, 3/4, 4/5}
    {
        \draw[thick] (\x,0) node (x\i) {} (\x+.5,1) node (y\i) {};
        \draw (x\i) ++ (0,-\off) node[lStyle] {\footnotesize{$x_{\i}$}};
        \draw (y\i) ++ (0,\off) node[lStyle] {\footnotesize{$y_{\i}$}};
    }

    \draw (2,1.8) node (b){} (4.5,-.8) node (a) {};
    \draw (b) ++ (0,\off) node[lStyle] {\footnotesize{$b_1$}};
    \draw (a) ++ (0,-\off) node[lStyle] {\footnotesize{$a_2$}};
    \draw (x5) -- (y5) -- (x1) -- (y1) -- (x2) -- (y2) -- (x3) -- (y3) -- (x4) -- (y4) -- (x5);

    \draw[ultra thick] (b) -- (x1) -- (y5) -- (x5) (y2) -- (x3) -- (y3) -- (a);
\end{scope}

\end{tikzpicture}
    \caption{Three choices of $J$ in Case~2 when $b_2=y_3$. 
    \label{key1factor-fig3}}
\end{figure}
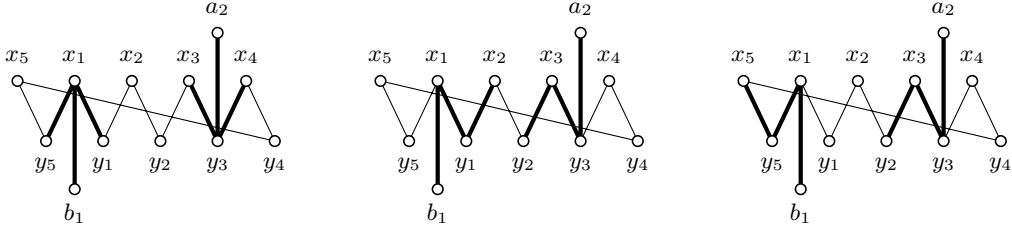

\begin{lem}
    Let $H$ be an $(8,4)$-bigraph with a matching $M$ of size 5 and a subgraph $C$ consisting of a vertex disjoint 6-cycle and 
    4-cycle (necessarily, $C$ is not vertex disjoint from $M$).
    There exists a subgraph $J$ of $H$ that contains a $2P_3$ from $C$ and a $2K_2$ from $M$ such that $H-E(J)$ has a 1-factor.
    \label{key1factorB-lem}
\end{lem}

\begin{proof}
   We denote the 6-cycle by $x_1y_1x_2y_2x_3y_3$ and the 4-cycle by $x_4y_4x_5y_5$.
    If $M$ contains an edge $e$ of the 4-cycle, then we let $E(J):=\{x_4y_4,y_4x_5,x_5y_5,y_5x_4,e'\}$, where $e'$ is an arbitrary edge
    of $M-e$.  Now $H-E(J)$ 
    contains a 1-factor by Corollary~\ref{matching-inc-cor}(2).  So assume this doesn't happen.  Similarly, if
    $M$ contains an edge $e_1$ incident to $\{x_4,x_5\}$ and an edge $e_2$ incident to $\{y_4,y_5\}$, then we let 
    $E(J):=\{x_4y_4,y_4x_5,x_5y_5,y_5x_4,e_1,e_2\}$; again, $H-E(J)$ 
    contains a 1-factor by Corollary~\ref{matching-inc-cor}(2), so we assume
    this doesn't happen.

    \textbf{Case 1: $\bm{x_1y_2\in M}$.}  Let $E(J):=\{x_1y_1,x_1y_2, x_1y_3,x_2y_2,x_3y_2,e\}$ where $e\in M-x_1y_2$.
    If this choice of $J$ fails, then we must have $y_1x_3,y_3x_2\in E(H)$.  Assume so, and now let 
    $E(J):=\{x_1y_1,x_1y_2,x_2y_1,x_3y_2,x_3y_3, e\}$, where $e\in M-x_1y_2$.  
    Note that $y_1\in N_{H'}(x_3)\setminus N_{H'}(x_2)$
    and $d_{H'}(x_2)\ge 3$.  Similarly, $x_3\in N_{H'}(y_1)\setminus N_{H'}(y_3)$ and $d_{H'}(y_3)\ge 3$.  Thus, this choice
    of $J$ succeeds.

    \textbf{Case 2: $\bm{x_1y_1\in M}$.}  Let $E(J):=\{x_1y_1,x_2y_1,x_3y_2,x_3y_3,e\}$, where $e\in M-x_1y_1$.
    If this choice of $J$ fails, then $x_1y_2,x_2y_3\in E(H)$.  And, by symmetry, also $x_3y_1\in E(H)$. 
    Now instead, let $E(J):=\{x_1y_1,x_1y_3,x_2y_1,x_2y_2,x_3y_3,e\}$
    where $e\in M-x_1y_1$.
    Since we are not in Case 1 above, $x_2y_3, x_3y_1\notin M$.
    Hence, $\{x_1,x_2,x_3\}\subsetneq N_{H'}(\{y_1,y_2,y_3\})$ and $\{y_1,y_2,y_3\}\subsetneq N_H(\{x_1,x_2,x_3\})$.
    Thus, this choice of $J$ succeeds.

    \textbf{Case 3: $\bm{x_1b_1,y_ia_2\in M}$ for some $\bm{i\in\{1,2,3\}}$.}  By symmetry, we assume that $i\in\{1,2\}$.
    Let $E(J):=\{b_1x_1,x_1y_3,y_3x_3,y_1x_2,x_2y_2,y_ia_2\}$.  Now $d_{H'}(b_1)\ge 3$ and $x_1\in N_{H'}(y_1)\setminus N_{H'}(b_1)$.
    Similarly, $d_{H'}(a_2)\ge 3$ and $y_i\in N_{H'}(\{x_1,x_3\})\setminus N_{H'}(a_2)$.  Thus, this choice of $J$ succeeds.
    \bigskip

    We assume we are not in any case above.  By Pigeonhole, 
    at least two edges of $M$ are incident to vertices $x_i$ with
    $i\in\{1,2,3,4,5\}$; similarly, at least two edges of $M$ are incident to vertices $y_j$ with $j\in \{1,2,3,4,5\}$.
    By symmetry between $A$ and $B$, we assume we are in the following case.

    \textbf{Case 4: $\bm{M}$ contains edges $\bm{x_1b_1,x_3b_3}$ and also contains edges $\bm{a_4y_4,a_5y_5}$.}
    We allow the possibility that one or both edges in the second list also appears in the first. By symmetry, we assume that 
    $a_4y_4\ne x_1b_1$ and that $a_5y_5 \ne x_3b_3$.  We may also assume that $a_4y_5,a_5y_4\notin E(H)$; otherwise, 
    we let $E(J):=\{x_4y_4,x_4y_5,x_5y_4,x_5y_5,a_4y_4,a_5y_5\}$ and are done by Lemma~\ref{matching-inc-lem}(1).

    Now let $E(J):=\{b_1x_1,x_1y_1,x_1y_3, a_4y_4,x_4y_4,x_5y_4\}$; see the left of Figure~\ref{key1factor-fig4}.  
    Note that $d_{H'}(a_4)\ge 3$ and $y_5\in N_{H'}(x_5)\setminus N_{H'}(a_4)$.  
    Thus, if $J$ fails, it is because $|N_{H'}(\{b_1,y_1,y_3,y_4\})|=3$, which implies $b_1x_2,b_1x_3,y_1x_3\in E(H)$. 
    By symmetry, (simply relabeling the vertices of the 6-cycle) we also assume that $b_3x_2\in E(H)$.   But now
    $d_H(x_2)\ge 5$. So let $E(J):=\{b_1x_1,x_1y_1,y_1x_2, a_4y_4,x_4y_4,x_5y_4\}$; see the right of Figure~\ref{key1factor-fig4}.  
    Now $d_{H'}(x_2)\ge 4$ and again $d_{H'}(a_4)\ge 3$ and $y_5\in N_{H'}(x_4)\setminus N_{H'}(a_4)$, 
    so $H-E(J)$ contains a 1-factor by Corollary~\ref{matching-inc-cor}(1).
\end{proof}

    \begin{figure}[!h]
    \centering
\begin{tikzpicture}[xscale = .9, scale=.8, yscale=-1]
    \def\off{.4cm}
\tikzset{every node/.style=uStyle}
\tikzstyle{gStyle}=[shape = circle, minimum size = 4pt, inner sep = 1pt,
outer sep = 0pt, fill=gray!50!white, semithick, draw]

\begin{scope}[xshift = -4in]
    \foreach \i/\x in {1/1, 2/2, 3/3, 4/6, 5/7}
    {
        \draw[thick] (\x,0) node (x\i) {} (\x+.5,1) node (y\i) {};
        \draw (x\i) ++ (0,-\off) node[lStyle] {\footnotesize{$x_{\i}$}};
        \draw (y\i) ++ (0,\off) node[lStyle] {\footnotesize{$y_{\i}$}};
    }

    \draw (.5,1) node (b1){} (4.5,1) node (b3) {};
    \draw (b1) ++ (0,\off) node[lStyle] {\footnotesize{$b_1$}};
    \draw (b3) ++ (0,\off) node[lStyle] {\footnotesize{$b_3$}};
    \draw (5,0) node (a4){} (8,0) node (a5) {};
    \draw (a4) ++ (0,-\off) node[lStyle] {\footnotesize{$a_4$}};
    \draw (a5) ++ (0,-\off) node[lStyle] {\footnotesize{$a_5$}};
    \draw (x1) -- (y1) -- (x2) -- (y2) -- (x3) -- (y3) -- (x1) (x4) -- (y4) -- (x5) -- (y5) -- (x4) (x3) -- (b3) (y5) -- (a5);
    \draw[dashed] (y4) -- (a5) (y5) -- (a4);

    \draw[ultra thick] (b1) -- (x1) -- (y1) (x1) -- (y3) (a4) -- (y4) -- (x4) (y4) -- (x5);
\end{scope}

\begin{scope}[xshift = 0in]
    \foreach \i/\x in {1/1, 2/2, 3/3, 4/6, 5/7}
    {
        \draw[thick] (\x,0) node (x\i) {} (\x+.5,1) node (y\i) {};
        \draw (x\i) ++ (0,-\off) node[lStyle] {\footnotesize{$x_{\i}$}};
        \draw (y\i) ++ (0,\off) node[lStyle] {\footnotesize{$y_{\i}$}};
    }

    \draw (.5,1) node (b1){} (4.5,1) node (b3) {};
    \draw (b1) ++ (0,\off) node[lStyle] {\footnotesize{$b_1$}};
    \draw (b3) ++ (0,\off) node[lStyle] {\footnotesize{$b_3$}};
    \draw (5,0) node (a4){} (8,0) node (a5) {};
    \draw (a4) ++ (0,-\off) node[lStyle] {\footnotesize{$a_4$}};
    \draw (a5) ++ (0,-\off) node[lStyle] {\footnotesize{$a_5$}};
    \draw (x1) -- (y1) -- (x2) -- (y2) -- (x3) -- (y3) -- (x1) (x4) -- (y4) -- (x5) -- (y5) -- (x4) (x3) -- (b3) (y5) -- (a5);
    \draw (x2) -- (b1) (x2) -- (y3) (x2) -- (b3);
    \draw[dashed] (y4) -- (a5) (y5) -- (a4);

    \draw[ultra thick] (b1) -- (x1) -- (y1) (y1) -- (x2) (a4) -- (y4) -- (x4) (y4) -- (x5);
\end{scope}

\end{tikzpicture}
    \caption{Two choices of $J$ in Case~4. 
    \label{key1factor-fig4}}
\end{figure}
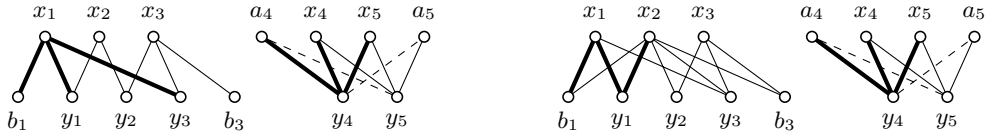


\section{Open Problems}
\label{open-sec}

We suggest the following open problems for further study.  For a graph class $\G$, 
let $\chisl(\G):=\max_{G\in \G}\chisl(G)$ and let $\chisc(\G):=\max_{G\in \G}\chisc(G)$.
Let $\P_g$ denote the class of planar graphs with girth at least $g$ and, for each positive real $a$, 
let $\G_a$ denote the class of graphs $G$ with $\mad(G)<a$.
\begin{enumerate}
    \item For each integer $g\in\{3,4\}$, determine $\chisl(\P_g)$ and $\chisc(\P_g)$.  Our main results are that $\chisc(\P_3)\le 8$
        and $\chisc(\P_4)\le 5$ and $\chisc(\P_5)\le 4$.  In~\cite[Theorem~4.1]{CCvBZ}, the authors present constructions showing 
        that $\chisl(\P_g)=\chisc(\P_g)=4$ for all $g\ge 5$.
    \item 
        For each positive real number $a$, determine $\chisl(\G_a)$ and $\chisc(\G_a)$.
        (In fact our proof that $\chisc(\P_4)\le 5$ actually proves that $\chisc(\G_4)\le 5$.)
        Conversely, given a positive integer $k$, determine the maximum values $a(k)$ and $b(k)$ such that $\chisl(\G_{a(k)})\le k$
        and $\chisc(\G_{b(k)})\le k$.
        Note that $a(k)\ge b(k)$ for all $k$, since every bad $k$-assignment gives rise to a bad $k$-cover.
        Here is a perhaps easier question: What are $\limsup_{k\to\infty}a(k)/k$ and $\limsup_{k\to\infty}b(k)/k$?
        Corollary~\ref{degen-cor} shows that $b(k)/k\ge 1/2$ for all $k$.  We can improve this somewhat as 
        follows.  

        For each integer $k\ge 3$, we have $b(2k-1)\ge k+k/(k+1)$.  Consider a graph $G$ and a $(2k-1)$-cover $(\vec{D},\sigma)$
        such that $G$ has no $(\vec{D},\sigma)$-packing, but every proper subgraph does.  By Corollary~\ref{degen-cor}, we 
        know $\delta(G)\ge k$.  We use discharging with initial charge $d(v)$ for all $v$.
        We use a single discharging rule: each $k$-vertex takes $1/(k+1)$ from each neighbor.
        Suppose $d(v)=k$.  By Lemma~\ref{k,k+1-lem}, we know $d(w)\ge k+2$ for all $w\in N(v)$.  Now $v$ finishes with charge 
        $k+k/(k+1)$.  If $d(v)=k+1$, then $v$ starts and finishes with charge $k+1$.  Finally, suppose $d(v)\ge k+2$.  Now $v$
        finishes with at least $d(v)-d(v)/(k+1) = d(v)k/(k+1)\ge (k+2)k/(k+1) = k + k/(k+1)$.  Thus, $\mad(G)\ge k+k/(k+1)$.

    \item Determine $\chisl(\B)$ and $\chisc(\B)$, where $\B$ denotes the class of all planar bipartite graphs.  
        We conjecture that $\chisl(\B)\le 4$, and also ask whether $\chisc(\B)\le 4$.  If the latter is true, then it is best
        possible, as shown by Example~\ref{ex:cycles} for (arbitrarily long) even cycles.

    \item For each surface $S$, let $\G_S$ denote the set of all graphs embeddable in $S$.  For every surface $S$, determine 
        $\chisl(\G_S)$ and $\chisc(\G_S)$.  Let $h(S)$ denote the \emph{Heawood number} of $S$, which is the order of the 
        largest complete graph
        embeddable in $S$.  It was famously proved (with the upper bound by Heawood, and the lower bound comprised of work by 
        various groups, most notably including Ringel and Youngs) that $\chi(\G_S)=\chi_{\ell}(\G_S)=h(S)$, for all surfaces except
        the plane and the Klein bottle.  In fact, it was later proved that if $G$ embeds in $S$ and $G$ has no subgraph $K_{h(S)}$,
        then $\chi_{\ell}(G)\le h(S)-1$.  In~\cite{CCvBDK}, it was proved that $\chisl(K_t)=t$ for all $t$.  And Yuster 
        observed~\cite{yuster}, based on a construction of Catlin~\cite{catlin}, that $\chisc(K_t)\ge t+1$ when $t$ is odd 
        (which is conjectured to be sharp).  Regardless of the precise value of $\chisc(K_t)$, 
        is natural to ask the following.
        For which surfaces $S$ do we have $\chisl(\G_S)=\chisl(K_{h(S)})$?  
        And for which do we have $\chisc(\G_S)=\chisc(K_{h(S)})$?  
        For which surfaces do we have $\chisl(G)\le h(S)-1$ for all $G\in \G_S$ that 
        do not contain $K_{h(S)}$ as a subgraph?  What is true in the analogous situation for $\chisc(G)$?
\end{enumerate}

\section*{Acknowledgments}
Thanks to two anonymous referees for helpful feedback.  In particular, one referee read the paper very carefully,
catching numerous inaccuracies and suggesting various ways to improve clarity.
\scriptsize{

}

\end{document}